\documentclass[11pt,reqno,tbtags]{amsart}

\usepackage[usenames,dvipsnames]{xcolor}
\RequirePackage[OT1]{fontenc}

\usepackage[colorlinks,citecolor=blue,urlcolor=blue,linkcolor=blue,linktocpage=true]{hyperref}

\usepackage[]{amsmath}
\usepackage{amssymb,amsthm,amsfonts,amsbsy,latexsym,amsxtra}

\usepackage{graphicx}

\usepackage{todonotes}
\usepackage{appendix}
\usepackage{url}

\usepackage[numbers, sort&compress]{natbib} 

\usepackage[usenames,dvipsnames]{xcolor}              
\usepackage{graphicx}
\usepackage{bbm,bm}
\usepackage{comment}
\usepackage{mathtools}
\usepackage{placeins}
\usepackage[labelfont={sl}]{caption}
\usepackage[labelfont={normalfont}]{subcaption}

\usepackage{amssymb,amsfonts,amsmath,amsthm,amscd,dsfont,mathrsfs}
\usepackage{url}
\usepackage{enumerate}
\usepackage{hyperref}

\numberwithin{equation}{section}

\newtheorem{theorem}{Theorem}[section]

\newtheorem{prop}[theorem]{Proposition}
\newtheorem{lemma}[theorem]{Lemma}

\theoremstyle{definition}

\newcommand{\E}{\mathbf{E}}
\newcommand{\R}{\mathbb{R}}
\newcommand{\Z}{\mathbb{Z}}
\newcommand{\N}{\mathbb{N}}

\newcommand{\eps}{\varepsilon}
\newcommand{\Prob}[1]{\mathbf P\{#1\}}

\usepackage{color}

\begin{document}
 	
 	\title[Large Degrees in Scale-Free Random Graphs]{Large Degrees in Scale-Free Inhomogeneous Random Graphs}
 	
 	\author[C.\ Bhattacharjee]{Chinmoy Bhattacharjee}
	\address{Department of Mathematics,
	University of Luxembourg}
 	\email{chinmoy.bhattacharjee@uni.lu}

 	\author[M.\ Schulte]{Matthias Schulte}
	\address{Institute of Mathematics,
	Hamburg University of Technology}
 	\email{matthias.schulte@tuhh.de}

 	\date{\today}
 	
 	


\begin{abstract} We consider a class of scale-free inhomogeneous random graphs, which includes some long-range percolation models. We study the maximum degree in such graphs in a growing observation window and show that its limiting distribution is Frechet. We achieve this by proving convergence of the underlying point process of the degrees to a certain Poisson process. Estimating the index of the power-law tail for the typical degree distribution is an important question in statistics. We prove consistency of the Hill estimator for the inverse of the tail exponent of the typical degree distribution.
\end{abstract}

\maketitle

\section{Introduction}

In this paper we study the behaviour of the large degrees for a certain class of scale-free inhomogeneous random graphs. In particular, we investigate the asymptotics for the maximum degree in such random graphs and prove consistency of the Hill estimator.

Large graphs with a highly non-trivial structure, so-called complex networks, arise in many different fields, ranging from natural to social sciences. Prominent examples of complex networks include the internet, the World Wide Web and social networks like Facebook, Twitter and so on. A key observation from network science is that many real-world complex networks show a scale-free or power-law behaviour, which essentially means that the proportion of vertices with degree greater than some positive integer $k$ is proportional to $k^{-\gamma}$ for some positive exponent $\gamma$ for large values of $k$. Thus, the degrees in a scale-free network have a large variability and there are a few vertices, known as hubs, with very large degrees. For more details on scale-free complex networks, we refer the reader to e.g.\ \cite{H16,Voitalov18}.

In order to model complex networks, one typically considers random graphs which share similar characteristics, in particular, the scale-free property. While the empirical degree distribution is a natural object to study in complex networks, for random graphs, it is often easier to deal with the typical degree, the degree of a uniformly chosen vertex. Since for random graphs, the empirical degree distribution is usually similar to the distribution of the typical degree, one is led to consider random graph models in which the typical degree $D$ satisfies
$$
\mathbf{P}\{D > k\} \sim c k^{-\gamma} \quad \text{as} \quad k\to\infty 
$$
for some positive constants $c$ and $\gamma$. In the sequel, we consider a more general class of degree distributions by allowing that
\begin{equation}\label{eqn:Degree}
\mathbf{P}\{D>u\}=\ell(u) u^{-\gamma}, \quad u> 0,
\end{equation}
where $\ell: [0,\infty)\to[0,\infty)$ is a slowly varying function, which means that for all $a>0$, $\lim_{u\to\infty} \ell(au)/\ell(u)=1$. We call a random graph scale-free if \eqref{eqn:Degree} is satisfied. 

Two random graphs which are ubiquitous in the literature are the Erd\H{o}s-R\'enyi graph and the random geometric graph. In the Erd\H{o}s-R\'enyi graph the vertices are given by $[n]=\{1,\hdots,n\}$ and each pair of distinct vertices is connected by an edge with probability $p\in(0,1)$ independently from the others. The random geometric graph, on the other hand, is obtained by taking the points of a stationary Poisson process $X$ in $\mathbb{R}^d$ with unit intensity as vertices and connecting two distinct vertices $x$ and $y$ by an edge whenever $|x-y|\leq r$ for a fixed parameter $r>0$. For the Erd\H{o}s-R\'enyi graph with $p=\lambda/n$ for $\lambda>0$, it is easy to see that the distribution of the typical degree converges weakly to a Poisson distribution with mean $\lambda$ as $n\to\infty$. In case of the random geometric graph, the typical degree (whose exact definition requires the Palm distribution) is $\operatorname{Poisson}(v_d r^d)$ distributed, where $v_d$ is the volume of the $d$-dimensional unit ball. So neither the Erd\H{o}s-R\'enyi graph, nor the random geometric graph is scale-free.

It is however possible to construct scale-free random graphs in a similar way as the Erd\H{o}s-R\'enyi graph or the random geometric graph by using non-negative independent and identically distributed (i.i.d.) weights whose distribution function $F(\cdot)$ is of the form
$$
1-F(u)= u^{-\beta} L(u), \quad u > 0,
$$ 
with a slowly varying function $L(\cdot)$ and $\beta>0$. A random variable with such a distribution function is said to have a regularly varying tail distribution with exponent $\beta$. An example of a scale-free random graph is the Norros-Reittu model introduced in \cite{NR06}, where one takes the vertices $[n]$, which are equipped with i.i.d.\ weights $(W_x)_{x\in[n]}$ distributed according to $F(\cdot)$. One then joins $E_n\{x,y\}=E_n\{y,x\}$ many edges between $x,y\in[n]$, where $(E_n\{x,y\})_{1\leq x\leq y \leq n}$ are independent and
$$
E_n\{x,y\}=_d\operatorname{Poisson}\left(W_{x} W_y/\sum_{z\in [n]} W_z\right).
$$
The resulting random multigraph with self-loops is a scale-free counterpart to the Erd\H{o}s-R\'enyi graph, where the typical degree has the same tail behaviour as the weights. Another such model is the so-called scale-free continuum percolation model considered in \cite{DW18}, which has the points of the stationary Poisson process $X$ as its vertices with associated i.i.d.\ weights $(W_x)_{x\in X}$ distributed according to $F(\cdot)$. An edge between distinct $x,y\in X$ is drawn independently with probability
$$
p_{xy}=1-\exp\left(-\lambda\frac{W_xW_y}{|x-y|^\alpha} \right)
$$  
with $\lambda,\alpha>0$. The resulting random graph is scale-free and can be seen as a scale-free counterpart to the random geometric graph.

In this paper, we study a general class of scale-free random graphs, which includes the Norros-Reittu model and the scale-free continuum percolation model. All these models are constructed using i.i.d.\ weights with a regularly varying tail distribution, such that vertices with large weights tend to have large degrees. Such random graph models are also called inhomogeneous random graphs. Some of the random graph models considered here, such as the scale-free continuum percolation model, are so-called long-range percolation models. In these models, the geometry of the underlying space becomes crucial since the probability that two vertices are connected by an edge depends on their distance as well.   

The first aim of this paper is to study the large degrees for a general class of scale-free inhomogeneous random graphs. Since some of them are infinite, we only consider the degrees of vertices within some finite observation window. For increasing observation windows, we prove that the maximum degree converges, after a rescaling, to a Frechet-distributed random variable. More generally, we establish the convergence of the rescaled degree sequences of the random graphs to a certain Poisson process. The behaviour of the maximum degree for the random graphs considered in this paper is significantly different from that of the maximum degree of an Erd\H{o}s-R\'enyi graph or a random geometric graph, which is concentrated with a high probability on at most two consecutive numbers (see \cite[Theorem 3.7]{Bollobas} and \cite[Theorem 6.6]{P03} or \cite{M08}, respectively). More recently, such a concentration phenomenon was also shown for the maximum degree of a Poisson-Delaunay graph in \cite{BC18}.

The second aim of this paper is to estimate the exponent $\gamma$ in \eqref{eqn:Degree}, an important statistical question (see \cite{Voitalov18} for a discussion of the problem and several estimators). A standard approach from extreme value theory to estimate the tail exponent of a regularly varying distribution is using the so-called \textit{Hill estimator}. For non-negative i.i.d.\ random variables $(X_i)_{1\le i \le n}$ with their common distribution having a regularly varying tail with exponent $\gamma>0$, the Hill estimator for $1/\gamma$ based on the $k$ upper-order statistics is defined as 
$$
H_{k,n}:=\frac{1}{k}\sum_{i=1}^k \log \frac{X_{(i)}}{X_{(k+1)}},
$$
where $X_{(1)} \ge X_{(2)}\ge \dots \ge X_{(n)}$ are the order statistics corresponding to the data $(X_i)_{1\le i \le n}$. It is well-known (see e.g.\ \cite[Theorem 4.2]{Re07}) that if $1 \le k=k_n < n$ is chosen so that $k_n\to \infty$ and $k_n/n \to 0$ as $n \to \infty$, then
$$H_{k_n,n} \xrightarrow{\mathbf{P}} \frac{1}{\gamma} \quad \text{as $n \to \infty$},$$
where $\xrightarrow{\mathbf{P}}$ denotes convergence in probability. In order to estimate the tail exponent $\gamma$ of the degree distribution, one can treat the degrees as i.i.d.\ random variables and apply the Hill estimator. This was proposed in \cite{Clausetetal2009} together with a minimum distance selection procedure to choose the parameter $k$ (see \cite{BhattacharyaChenHofstadZwart, Drees2018} for some theoretical results on this selection procedure). However, since the degrees of the vertices are usually not independent, the consistency result above is not applicable in the context of random graphs. An exception is the configuration model (see e.g.\ \cite{H16}), where the degrees are i.i.d.\ by construction. In \cite{RW19,RW18} the consistency of the Hill estimator for linear preferential attachment models is shown, see \cite{Drees2018,WanWangDavisResnick} for related simulation studies. But apart from that, we are not aware of any other results in this direction. We address this question by establishing consistency of the Hill estimator for a large class of scale-free inhomogeneous random graphs. We believe that such consistency results are important for network science since they justify to some extent the use of the Hill estimator for the analysis of real-world complex networks. 

The main challenge in proving our results is that the degrees of the vertices are not independent. We resolve this problem by approximating the degrees of the vertices in terms of their weights, which are a family of i.i.d.\ random variables. As a byproduct of our approach, we obtain a one-to-one correspondence between the vertices with large degrees and those with large weights.

The rest of the paper is organised as follows. In Section~\ref{sec:2}, we introduce the class of scale-free inhomogeneous random graphs we are interested in and provide some examples, before presenting our main results. Section~\ref{sec:proofs} is devoted to the proof of the abstract results, which are applied to some particular random graph models in Section~\ref{sec:app}.

\subsection*{Notation} Throughout, we denote by $(k_n)_{n \in \N}$ an \textit{intermediate sequence}, that is a positive integer sequence with $k_n < n$ for all $n \in\N$, $k_n \to \infty$ and $k_n/n \to 0$ as $n \to \infty$, while $\mathbf{1}$ stands for the sequence $(1,1,1,\cdots)$. We denote by $M_+((0,\infty])$ the set of all non-negative Radon measures on $(0,\infty]$ equipped with the \textit{vague} topology and by $M_p((0,\infty])$ the subset of point measures on $(0,\infty]$. Let $\delta_{x}$ stand for the Dirac delta measure with point mass at $x$. 
We denote by $\xrightarrow{d}$ and $\xrightarrow{\mathbf{P}}$ convergence in distribution and in probability, respectively. For two real valued functions $f(\cdot)$ and $g(\cdot)$ on $\R$, we write $f(x) \sim g(x)$ as $x \to \infty$ when $\lim_{x \to \infty}f(x)/g(x)=1$. For a random variable $Z$ with distribution function $F(\cdot)$ we write $Z\sim F$. We denote by $\mathrm{RV}_\rho$ the set of real-valued functions on $\R$ which are regularly varying with index $\rho\not=0$.

\section{Main results}\label{sec:2} We start with a short informal description of our model in a general setup and provide some examples. We consider a sequence of random graphs $(G_n)_{n\in \N}$. For each $n \in \N$, the set of vertices $V_n$ of $G_n$ is an at most countably infinite set (possibly random) in some space $S$. The vertices of $G_n$ are independently marked with weights $(W_x)_{x\in V_n}$ which are non-negative i.i.d.\ random variables with common distribution $F(\cdot)$. The probability that two vertices $x,y\in V_n$ are connected by an edge depends on their weights $W_x$ and $W_y$ (and sometimes on their distance). Throughout this paper, we will be interested in the following random graph models that are of the form described above:\\

\noindent\textbf{I. Scale-free percolation models on $\Z^d$.} We consider an inhomogeneous random graph model for long-range percolation on the lattice $\Z^d$, $d \in \N$, as introduced in \cite{DHH13}. We let $G_n=G$ for all $n \in \N$, where the random graph $G$ with vertex set $\mathbb{Z}^d$ is constructed as follows. Given the weights $(W_x)_{x\in\mathbb{Z}^d}$, for distinct $x,y \in \mathbb{Z}^d$, we add the edge $\{x,y\}$ in $G$ independently of the other edges with probability
\begin{equation}\label{eq:ConPr}
p_{xy}=1-\exp\left(-\lambda \frac{W_x W_y}{|x-y|^\alpha}\right)
\end{equation}
for parameters $\alpha,\lambda \in (0,\infty)$ and $|\cdot|$ denoting the Euclidean norm. The idea behind \eqref{eq:ConPr} is that the probability of an edge appearing in the graph is decreasing in the distance between the endpoints, while large weights at either endpoint increase the probability. \\

\noindent\textbf{II. Ultra-small scale-free geometric networks on $\Z^d$.} This model was considered in \cite{Y06} in the context of long-range percolation. We again consider the vertex sets $V_n=\Z^d$, $n \in \N$, and independent weights $(W_x)_{x \in \Z^d}$. In this example, we choose the weights to have a particular distribution, namely, $W_x= U_x^{-1/\beta}$ for some $\beta>0$, where $(U_x)_{x \in \Z^d}$ is a collection of i.i.d.\ uniform random variables on the interval $[0,1]$. We take $G_n=G$ for all $n \in \N$, where $G$ is defined as follows. Given the weights, we add the edge $\{x,y\}$ for distinct $x, y \in \Z^d$ independently with probability
\begin{equation*}
p_{xy}=\mathds{1}\left\{ \min \left\{U_x^{-1/\beta} ,U_y^{-1/\beta}\right\} \ge |x-y| \right\}.
\end{equation*} \\

\noindent\textbf{III. Scale-free continuum percolation models.} The heterogeneous random connection model (RCM), which was introduced in \cite{DW18}, is a continuum space analogue of the random graph model {\bf I}. Let $X$ be a stationary Poisson process in $\R^d$ with unit intensity and i.i.d.\ marks $(W_x)_{x \in X}$. Given $X$ and $(W_x)_{x \in X}$, we take $V_n=X$ and $G_n=G$ for all $n \in \N$, where in $G$, we join two points $x\not=y \in X$ by an edge independently with probability $p_{xy}$ given by \eqref{eq:ConPr}.\\

\noindent\textbf{IV. Norros-Reittu model.} Unlike the models {\bf I}--{\bf III}, this model from \cite{NR06} does not depend on the geometry of the space and allows multiple edges and self-loops. Moreover the graphs $(G_n)_{n \in \N}$ are not identical but evolve with $n$. Let  $(W_x)_{x \in \N}$ be a sequence of i.i.d.\ weights. For $n \in \N$, the vertex set $V_n$ is taken to be $[n]=\{1,\hdots,n\}$ with assigned weights $(W_x)_{x \in [n]}$. Define $L_n=\sum_{k=1}^{n}W_k$. For each $n \in \N$, given the weights $(W_x)_{x \in [n]}$, we construct the random graph $G_n$ on the vertex set $[n]$ by joining $E_n\{x,y\}=E_n\{y,x\}$ many edges between the vertices $x,y \in [n]$, where $(E_n\{x,y\})_{1\leq x \leq y\leq n}$ are independent with
$$E_n\{x,y\} \sim \mathrm{Poisson}(W_xW_y/L_n).$$

\noindent\textbf{V. Chung-Lu model.} The Chung-Lu random graph model with deterministic weights was introduced in \cite{CL02a,CL02b}. Here, we consider a version with random weights, see \cite{EskerHofstadHooghiemstra}. Let $(W_x)_{x \in \N}$ be i.i.d.\ weights.
For $n \in \N$, we take the vertex set $V_n=[n]$ with associated weights $(W_x)_{x \in [n]}$. To construct the graph $G_n$, given the weights $(W_x)_{x \in [n]}$, we join an edge between the vertices $x,y \in [n]$ independently with probability
$$p_{xy}=\min\left\{\frac{W_xW_y}{L_n}, 1\right\}$$
where $L_n=\sum_{k=1}^{n}W_k$. This model is closely related to the model {\bf IV}; it replaces the Poisson random variable for the number of edges between two vertices in {\bf IV} by a related Bernoulli random variable, thus getting rid of the multiple edges while still allowing self-loops. We will see in Section~\ref{sec:app} that one can indeed construct a natural coupling of these two models.\\

We now turn to describing our general random graph model in more detail. For $n\in \N$, let $(D_{n,x})_{x \in V_n}$ denote the degree sequence of the graph $G_n$. We note here that for the degree of a vertex in a graph with multiple edges or self-loops, multiple edges are counted with their multiplicities, while each self-loop contributes only one. The main results in this paper require some assumptions on the underlying graphs $(G_n)_{n\in \N}$, their degrees and the weights.

Recall, a function $U:[0,\infty) \to [0,\infty)$ is said to be \textit{regularly varying} with \textit{index} $\rho \in \R$ (we write $U\in \mathrm{RV}_\rho$) if for any $a>0$,
$$\lim_{t \to \infty}\frac{U(at)}{U(t)}=a^\rho.$$
If $\rho=0$, we call $U$ \textit{slowly varying}. An equivalent and often more insightful way of representing a regularly varying function $U(\cdot)$ with index $\rho$ is to write
$$U(t)=t^\rho V(t), \quad t> 0,$$
where $V(\cdot)$ is slowly varying. We restrict ourselves to the following class of distributions for the weights.
\begin{enumerate}
	\item[(A1)] The common weight distribution $F(\cdot)$ has a regularly varying tail with \textit{exponent} $\beta>0$, i.e., $1-F \in \mathrm{RV}_{-\beta}$.
\end{enumerate}
This means that
\begin{equation}\label{eq:WeDis}
1-F(w)=w^{-\beta}L(w), \quad w>0,
\end{equation}
where $L(\cdot)$ is a slowly varying function.
As mentioned in the introduction, in real-world complex networks, it is often observed that they are scale-free. As we will see in Theorem~\ref{thm:deg}, the above choice of distribution for the weights is crucial to ensure that the degree distribution has a power-law tail.

Our remaining assumptions concern the behaviour of the degrees of $(G_n)_{n\in \N}$.
\begin{enumerate}
	\item[(A2)] For each $n \in \N$, there exists a random vector $(D_n,W)$ with $W \sim F$ as in \eqref{eq:WeDis} such that for any measurable set $A \subseteq \N_0 \times [0,\infty)$ and any $B\subseteq S$ with $\E|V_n\cap B|<\infty$,
	$$\E \sum_{x\in V_n\cap B} \mathds{1}\{(D_{n,x},W_x) \in A\}=\E| V_n\cap B| \Prob{(D_n,W) \in A}.$$
\end{enumerate}
We will call the random variables $D_n$ and $W$ the typical degree and weight, respectively. Note that when the vertex sets $(V_n)_{n \in \N}$ are non-random and for each $n\in\N$, $(D_{n,x},W_x)=_d (D_{n,y},W_y)$ for $x , y \in V_n$, the assumption (A2) is trivially satisfied with $(D_n,W)= (D_{n,x_n}, W_{x_n})$ for some fixed $x_n\in V_n$. This applies to the models {\bf I}, {\bf II}, {\bf IV} and {\bf V}. For the model {\bf III}, since the vertex set is a Poisson process, we cannot take a fixed vertex as typical vertex. Instead, we add the origin $0$ to the vertex set with independently assigned weight $W_0 \sim F$ and extend the graph in a natural way to include $0$ as a vertex. This new measure after adding $0$ to the Poisson process happens to be the Palm measure corresponding to the original probability measure; for a detailed exposition on Palm theory, see e.g.\ \cite[Chapter 12]{DV88}. Hence, we can take the typical degree-weight pair to be $(D_{n,0},W_0)$ where $D_{n,0}$ is the degree of $0$ in the extended graph.

We make the following assumptions about the typical degree $D_n$ and the typical weight $W$. By $\E_W$ we denote the conditional expectation with respect to $W$ throughout this paper.
\begin{enumerate}
	\item[(A3)] There exist positive universal constants $\xi,p,C$ with a sequence of positive real numbers $\xi_k \to \xi$ as $k \to \infty$ and $\varrho \in (0,1]$ such that 
	$$
	|\E_W D_n - \xi_n W^p| \le C\max\left\{1,W^{p(1-\varrho)}\right\} \quad \mathbf{P}\text{-a.s.}
	$$
	for all $n\in\N$.
	\item[(A4)] For each $m \in \N$, there exist non-negative constants $a_m$ and $C_m$ depending only on $m$ such that
	$$\E_W |D_n-\E_W D_n|^{m}\le a_m W^{\frac{mp}{2}} + C_m  \quad \mathbf{P}\text{-a.s.}$$
	for all $n\in\N$. Here, $p$ is as in (A3).
\end{enumerate}

In Section \ref{sec:app}, we will show that the random graph models {\bf I}--{\bf IV} discussed above satisfy the assumptions (A2)-(A4) if the weights are chosen according to (A1) and some mild assumptions on the model parameters are satisfied. For model {\bf V}, it turns out that (A3) does not hold. We instead use a coupling with model {\bf IV} to prove our results for this model.

In this paper, we are interested in the large degrees of the random graphs $(G_n)_{n\in \N}$. Since the graph $G_n$ can have infinitely many vertices, we will only consider the degrees of finitely many vertices within an observation window and let the window grow with $n$. For this, throughout the sequel, we let $(S_n)_{n\in \N}$ be a sequence of deterministic sets such that for $\Delta_n=V_n\cap S_n$,
\begin{equation}\label{eqn:Delta}
\lim_{n\to\infty} \frac{\E |\Delta_n|}{n}=1 \quad \text{and} \quad \frac{|\Delta_n|}{n} \xrightarrow{\mathbf{P}} 1 \; \text{ as }\; n \to \infty,
\end{equation}
where $|\Delta_n|$ denotes the cardinality of $\Delta_n$. Recall $W \sim F$ from (A2) and the constant $p>0$ given by assumption (A3). We define the \textit{quantile function} $q(\cdot)$ for $W^p$ as
\begin{equation}\label{eq:qf}
q(t)=\inf\left \{x \ge 0:\Prob{W^p \le x} \ge 1-1/t\right \}, \quad t \ge 1.
\end{equation}
Throughout, $\mathcal{D}_n$ stands for the point process
\begin{equation*}
\mathcal{D}_n=\sum_{x \in \Delta_n} \delta_{\frac{D_{n,x}}{\xi_n q(n)}},
\end{equation*}
i.e., $\mathcal{D}_n$ is the collection of the degrees of vertices of $G_n$ that belong to $S_n$, up to a rescaling by $\xi_n q(n)$ with $\xi_n$ from assumption (A3).
Recall, we count multiple edges with multiplicities when considering the degree of an associated vertex, while each self-loop contributes only one to the degree of its vertex.

Throughout the whole paper, we fix $\gamma:=\beta/p$ where $\beta$ and $p$ are the constants from the assumptions (A1) and (A3), respectively. Let $\nu_\gamma$ denote the measure on $(0,\infty]$ given by 
$$
\nu_\gamma((a,b])=a^{-\gamma}-b^{-\gamma}
$$
for all $0<a<b \le \infty$. Also recall that $Z \sim {\rm Frechet}(\gamma)$ if $\Prob{Z \le z}=e^{-z^{-\gamma}}$ for all $z \ge 0$. The following result proves that the scaled degree sequence $\mathcal{D}_n$ asymptotically behaves like a Poisson process as $n \to \infty$.

\begin{theorem}\label{thm.PoiLim} Assume that (A1)-(A4) are satisfied and let $\eta_{\gamma}$ be a Poisson process with intensity measure $\nu_{\gamma}$. Then, as $n \to \infty$,
	\begin{equation}\label{ProConv}
	\mathcal{D}_n \xrightarrow{d} \eta_{\gamma} \quad \text{in } M_p((0,\infty]).
	\end{equation}
	In particular, as $n \to \infty$,
	\begin{equation}\label{eq:Frl}
	\xi_n^{-1}q(n)^{-1} \max_{x \in \Delta_n}D_{n,x} \xrightarrow{d} {\rm Frechet}(\gamma).
	\end{equation}
\end{theorem}

Since $\eta_\gamma$ as a point process on $[0,\infty]$ is not locally finite at zero, we exclude the origin in \eqref{ProConv} by considering the convergence only in $M_p((0,\infty])$. Note however, that we do include infinity in our underlying space. We consider the usual topology on $(0,\infty]$, i.e., the topology generated by intervals of the form $(a,b)$, $0 \le a<b<\infty$, and $(a,\infty]$ for $a \in [0,\infty)$. This naturally makes neighbourhoods of $\infty$ relatively compact, which is convenient for us, as a convergence like \eqref{eq:Frl} then follows as an immediate consequence of \eqref{ProConv}. For a detailed discussion on the space $M_p((0,\infty])$, its properties and vague convergence in this space, see \cite[Chapter 6]{Re07} or \cite[Chapter 3]{Re87}. We note that one can derive a similar result as \eqref{eq:Frl} for the $m$-th largest degree from \eqref{ProConv}.

Denote the order statistics for the degree sequence $(D_{n,x})_{x \in \Delta_n}$ by 
$$D_{(1)}(n) \ge D_{(2)}(n)\ge \dots \ge D_{(|\Delta_n|)}(n).$$
For random variables with regularly varying distribution tail, one is often interested in estimating the index of regular variation. In the i.i.d.\ case, one way to estimate this quantity is by the Hill estimator, as indicated in the introduction. Define the Hill estimator based on the $k \in \N$ upper order statistics of $(D_{n,x})_{x \in \Delta_n}$ as 
$$
H_{ k,n} = \frac{1}{k} \sum _ { i = 1 } ^ {k} \log \frac { D_{(i)}(n) } { D_{\left(k+1\right)}(n)}.
$$
In the i.i.d.\ case, the Hill estimator approximates the inverse of the tail exponent. It turns out that the same holds under the special dependency structure we have in the degree sequence. The following theorem proves consistency of the Hill estimator. Recall that a positive integer sequence $(k_n)_{n \in \N}$ is an intermediate sequence if $k_n < n$ for all $n \in\N$, $k_n \to \infty$ and $k_n/n \to 0$ as $n \to \infty$.

\begin{theorem}\label{thm:hillCons} Assume that (A1)-(A4) are satisfied and let $(k_n)_{n \in \N}$ be an intermediate sequence. Then
	$$H_{k_{n},n}\xrightarrow{\mathbf{P}} \frac{1}{\gamma} \quad \text{as }\; n \to \infty.$$
\end{theorem}

Our main results Theorems~\ref{thm.PoiLim} and \ref{thm:hillCons} require the general conditions given by (A1)-(A4) and apply to a broad class of random graphs. In Section~\ref{sec:app}, we apply these results to the models {\bf I}--{\bf V} described above and prove the following theorem showing that the conclusions of Theorems~\ref{thm.PoiLim} and \ref{thm:hillCons} hold for these random graphs under condition (A1) and appropriate assumptions on their parameters.

\begin{theorem}\label{thm:I-V}
	Assume (A1) is satisfied. Then the conclusions of Theorems~\ref{thm.PoiLim} and \ref{thm:hillCons} hold
	\begin{enumerate}[(a)]
		\item for models {\bf I} and {\bf III} with 
		$$\xi_n=\lambda^{d/\alpha}v_{d}\Gamma\left(1-\frac{d}{\alpha}\right)\E\left[W^{d/\alpha}\right],$$
		where $v_d$ denotes the volume of the unit ball in $\R^d$ and $\Gamma(\cdot)$ is the Gamma function, $S_n=[0,n^{1/d}]^d$, $n \in \N$, and $\gamma=\alpha\beta/d$ if $d <\min\{\alpha,\alpha \beta\}$,
		\item for model {\bf II} with $\xi_n=dv_d/(d-\beta)$, $S_n=[0,n^{1/d}]^d$, $n \in \N$, and $\gamma=\beta/(d-\beta)$ if $\beta <d$,
		\item for model {\bf IV} with $\xi_n=1$, $S_n=[n]$, $n \in \N$, and $\gamma=\beta$,
		\item for model {\bf V} with $\xi_n=1$, $S_n=[n]$, $n \in \N$, and $\gamma=\beta$ if $\beta>2$. 
	\end{enumerate}
\end{theorem}

\section{Proofs of Theorems~\ref{thm.PoiLim} and \ref{thm:hillCons}}\label{sec:proofs}

\subsection{Distribution of the typical degree}
We start by proving the following result, which shows that the typical degree $D_n$ of the graph $G_n$ satisfying (A1)-(A4) has a regularly varying distribution tail. For the models {\bf I}-{\bf IV} this is well-known (under some assumptions on the model parameters, see \cite{DHH13,Y06,DW18,NR06}), but our finding shows that this is a consequence of the abstract assumptions (A1)-(A4). Recall that $\gamma=\beta/p$ with $\beta$ and $p$ as in the assumptions (A1) and (A3), respectively.

\begin{theorem}\label{thm:deg}
	Assume that (A1)-(A4) are satisfied. Then for each $n\in\N$ the typical degree distribution has a regularly varying tail with exponent $\gamma$, i.e., for $t>0$,
	$$
	\Prob{D_n>t}=t^{-\gamma} \ell_n(t),
	$$
	where $\ell_n(\cdot)$ is slowly varying.
\end{theorem}
\begin{proof}
	Since by (A1)-(A2) the distribution of $\xi_n W^p$ has a regularly varying tail with exponent $\gamma$, it is enough to show that
	\begin{equation}\label{eq:DWeq}
	\lim_{t \to \infty}\frac{\Prob{D_n>t}}{\Prob{\xi_n W^p>t}}=1.
	\end{equation}
	Let $\mathbf{P}_W$ denote the conditional probability with respect to $W$. Recall $\varrho$ from (A3). Fix $\eps>0$ and $1>x>\max\{1/2,1-\varrho\}$ and let $t>0$. Define $u=u(W,t)=\max\{W^{xp},\varepsilon t\}$ and let $E_{u,W}$ denote the conditional probability $\mathbf{P}_W\{|D_n-\xi_n W^p|>u\}$. Then, we have
	\begin{equation*}
	\mathds{1}\{\xi_n W^p>t+u\} - E_{u,W} \le \mathbf{P}_W\{D_n>t\} \le \mathds{1}\{\xi_n W^p>t-u\} + E_{u,W} \quad \mathbf{P}\text{-a.s.},
	\end{equation*}
	and hence taking expectation,
	\begin{equation}\label{eq:DegIneq}
	\Prob{\xi_n W^p>t+u} - \E E_{u,W}  \le \Prob{D_n>t} \le \Prob{\xi_n W^p>t-u} + \E E_{u,W}.
	\end{equation}
	Using the Markov inequality in the first step and (A3)-(A4) in the second, for any $m \in \N$, we obtain that $\mathbf{P}\text{-a.s.}$,
	\begin{align*}
	E_{u,W} & \le u^{-m} \E_W|D_{n}-\xi_n W^p|^m \\
	&\le u^{-m}2^{m-1} \left(a_m W^{mp/2}+C_m+C^m\max\left\{1,W^{mp(1-\varrho)}\right\}\right) \\
	&\le 2^{m-1}\left[a_m u^{-m(1-1/(2x))}+C^m u^{-m(1-(1-\varrho)/x)}\right] +(C_m+C^m)2^{m-1}u^{-m},
	\end{align*} 
	where in the last step we have used that $u=\max\{W^{xp},\varepsilon t\}\ge W^{xp}$.
	Choosing $m$ large enough so that $m(1-x^{-1}\max\{1/2,1-\varrho\}) \ge 2\gamma$ and using that $u \ge \eps t$, this yields that for $t \ge 1/\varepsilon$,
	\begin{align*}
	\E E_{u,W}\le C'_{m}(\varepsilon t)^{-2\gamma}
	\end{align*}
	for some constant $C'_m>0$. Recalling that $\xi_n W^p$ has a regularly varying tail with exponent $\gamma$, it follows that
	\begin{equation}\label{eq:er}
	\lim_{t \to \infty}\frac{\E E_{u,W}}{\Prob{\xi_n W^p>t}}=0.
	\end{equation}
	On the other hand, since $W \sim F$ by (A2),  using $u(W,t) \le W^{xp}+\varepsilon t$ in the first step, for $t$ large enough we have
	\begin{equation*}
	\Prob{\xi_n W^p - u(W,t)>t} \ge \Prob{\xi_n W^p-W^{xp}>(1+\varepsilon)t}
	\ge \Prob{\xi_n W^p>(1+2\varepsilon)t}.
	\end{equation*}
	Again, by regular variation of the tail distribution of $\xi_n W^p$, we obtain
	\begin{equation}\label{eq:lb1}
	\liminf_{t \to \infty}\frac{\Prob{\xi_n W^p - u(W,t)>t}}{\Prob{\xi_n W^p>t}}
	\ge\liminf_{t \to \infty}\frac{\Prob{\xi_n W^p>(1+2\varepsilon)t}}{\Prob{\xi_n W^p>t}}=\left(1+2\varepsilon\right)^{-\gamma}.
	\end{equation}
	An identical argument shows that for $t$ large enough,
	\begin{equation*}
	\Prob{\xi_n W^p + u(W,t)>t} \le \Prob{\xi_n W^p+W^{xp}>(1-\varepsilon)t}
	\le \Prob{\xi_n W^p>(1-2\varepsilon)t},
	\end{equation*}
	which yields
	\begin{equation}\label{eq:ub1}
	\limsup_{t \to \infty}\frac{\Prob{\xi_n W^p + u(W,t)>t}}{\Prob{\xi_n W^p>t}}\le \left(1-2\varepsilon\right)^{-\gamma}.
	\end{equation}
	Combining \eqref{eq:DegIneq}, \eqref{eq:er}, \eqref{eq:lb1} and \eqref{eq:ub1}, we obtain
	$$\left(1+2\varepsilon\right)^{-\gamma} \le \liminf_{t \to \infty} \frac{\Prob{D_n>t}}{\Prob{\xi_n W^p>t}} \le \limsup_{t \to \infty} \frac{\Prob{D_n>t}}{\Prob{\xi_n W^p>t}} \le \left(1-2\varepsilon\right)^{-\gamma}.$$
	Since $\varepsilon>0$ was chosen arbitrarily, taking $\varepsilon \to 0$ now yields \eqref{eq:DWeq} concluding the proof.
\end{proof}

\subsection{Approximating the degrees in terms of the weights} A crucial step in proving Theorem \ref{thm.PoiLim} is to compare the degree sequence $(D_{n,x})_{x \in \Delta_n}$ with a proper transformation of the weight sequence $(W_x)_{x \in \Delta_n}$. The following theorem shows that the large degrees behave asymptotically like a power of the weights of the associated vertices. The quantile function $q(\cdot)$ was defined in \eqref{eq:qf}.

\begin{theorem}\label{thm:exp}
	Assume that (A1)-(A4) are satisfied and let $(t_n)_{n \in \N}$ be either an intermediate sequence or $\mathbf{1}$. Then for any $a>0$,
	$$
	\lim_{n \to \infty} \frac{1}{t_n} \E \sum_{x \in \Delta_n} \mathds{1}\{D_{n,x} \ge q(n/t_n) \xi_n a \ge \xi_n W_x^p\}=0
	$$
	and
	$$
	\lim_{n \to \infty} \frac{1}{t_n} \E \sum_{x \in \Delta_n} \mathds{1}\{\xi_n W_x^p \ge q(n/t_n)\xi_n a \ge D_{n,x}\} =0.
	$$
\end{theorem}

Before proceeding to the proof, let us first comment on this result. Consider the case when $t_n=1$ for all $n \in \N$. Then, the theorem above essentially states that the two sequences $(D_{n,x})_{x \in \Delta_n}$ and $(\xi_n W_x^p)_{x \in \Delta_n}$, when rescaled, look similar when $n$ is large. Since $(\xi_n W_x^p)_{x \in \Delta_n}$ is a collection of i.i.d.\ random variables, one can expect it to converge to a Poisson process, when appropriately transformed. In Lemma \ref{lem:welim}, we establish this fact. This along with Theorem \ref{thm:exp} implies that the point process formed by $(D_{n,x})_{x \in \Delta_n}$ can be approximated after rescaling by a suitable Poisson process, as given by Theorem~\ref{thm.PoiLim}.

We prepare the proof of Theorem \ref{thm:exp} by first noting some properties of $q(\cdot)$ and establishing two lemmas. It is not hard to see that $q(\cdot)$ is non-decreasing and that $q(t)$ diverges to infinity as $t \to \infty$. It follows from \cite[Theorem 3.6 and Remark 3.3]{Re07} that for all $c>0$, as $t \to \infty$,
\begin{equation}\label{eq:stail}
\Prob{W^p>cq(t)} \sim \frac{c^{-\gamma}}{t}.
\end{equation}

The next result shows that the function $q(\cdot)$ is regularly varying.

\begin{lemma}\label{lem:bsl} Assume (A1) is satisfied. Then the function $q(\cdot)$ is regularly varying with index $1/\gamma$.
\end{lemma}
\begin{proof}
	For a function $T:[0,\infty) \to [0,\infty)$, we write $T^{\leftarrow}(\cdot)$ for its generalised inverse, i.e.,
	$$
	T^{\leftarrow}(t)=\inf\{x \ge 0 : T(x) \ge t\}.
	$$
	Note from \eqref{eq:qf} that 
	$q(t)=\left(\frac{1}{1-F_p}\right)^{\leftarrow}(t)$,
	where $F_p(\cdot)$ is the distribution function of $W^p$. It is straightforward to check using (A1) that the function $h(x)=(1-F_p(x))^{-1}$ is non-negative, non-decreasing, unbounded and is an element of $\mathrm{RV}_{\gamma}$.
	Applying \cite[Proposition 2.6 (vi)]{Re07}, we obtain for all $a > 0$ that as $t \to \infty$,
	\begin{equation*}
	q(at)=\inf\{x \ge 0: h(x) \ge at\}=(a^{-1}h)^{\leftarrow}(t) \sim (a^{-1})^{-1/\gamma}h^{\leftarrow}(t)=a^{1/\gamma}q(t),
	\end{equation*}
	which proves the result.
\end{proof}

Recall for the following lemma that $\E_W$ denotes the conditional expectation with respect to $W$.

\begin{lemma}\label{lem:part}
	Assume that (A1)-(A4) are satisfied and let $(t_n)_{n \in \N}$ be either an intermediate sequence or $\mathbf{1}$. Then, for $a>0$ and $\varepsilon>0$,
	$$
	\limsup_{n \to \infty}\frac{n}{t_n}\Prob{ W^p < aq(n/t_n), |D_n-  \E_W D_n |>\varepsilon q(n/t_n)}=0
	$$
	and
	$$
	\limsup_{n \to \infty}\frac{n}{t_n}\Prob{ W^p > aq(n/t_n), |D_n-  \E_W D_n |>\varepsilon q(n/t_n)}=0.
	$$
\end{lemma}

\begin{proof}
	Let $\mathbf{P}_{W}$ denote the conditional probability with respect to $W$. Using the Markov inequality in the second step and (A4) in the third, one has that for any $m \in \N$,
	\begin{align*}
	&\Prob{ W^p < aq(n/t_n), |D_n-  \E_W D_n |>\varepsilon q(n/t_n)}\\
	& \le \E \Big[\mathds{1}\{W^p \le a q(n/t_n)\} \mathbf{P}_{W}\left\{ |D_n -\E_W D_n| >\varepsilon q(n/t_n)\right\}\Big]\\
	&\le \frac{1}{(\varepsilon q(n/t_n))^m}\E \left[\mathds{1}\{W^p \le a q(n/t_n)\}\E_W |D_n - \E_W D_n|^{m} \right]\\
	&\le \frac{1}{(\varepsilon q(n/t_n))^m}\E \left[\mathds{1}\{W^p \le a q(n/t_n)\}a_m W^{\frac{m p}{2}} + C_m\right].
	\end{align*}
	We choose $m > 2\gamma$. Since $W \sim F$ by (A2), it follows from \eqref{eq:WeDis} that $W^{\frac{mp}{2}}$ has a regularly varying tail with index $-2\gamma/m > -1$. Hence using Karamata's theorem (see e.g.\ \cite[Theorem 2.1]{Re07}) in the second step and \eqref{eq:stail} in the third, we have that as $n \to \infty$,
	\begin{align*}
	&\E \left[\mathds{1}\{W^p \le a q(n/t_n)\}W^{\frac{m p}{2}}\right] \le \int_0^{(aq(n/t_n))^{m/2}} \Prob{W^{\frac{mp}{2}}>t}dt\\
	&\sim \left(1-\frac{2\gamma}{m}\right)^{-1}(a q(n/t_n))^{m/2}\Prob{W^{p}>a q(n/t_n)}\\
	&\sim \left(1-\frac{2\gamma}{m}\right)^{-1}(a q(n/t_n))^{m/2} \frac{t_n a^{-\gamma}}{n}.
	\end{align*} 
	Thus, we obtain
	\begin{align*}
	&\limsup_{n \to \infty}\frac{n}{t_n}\Prob{ W^p < aq(n/t_n), |D_n-  \E_W D_n |>\varepsilon q(n/t_n)}\\
	&\le \limsup_{n \to \infty}\frac{n}{t_n(\varepsilon q(n/t_n))^m}\left[ a_m \left(1-\frac{2\gamma}{m}\right)^{-1}(a q(n/t_n))^{m/2} \frac{t_n a^{-\gamma}}{n}+C_m\right]=0,
	\end{align*}
	where in the last step, for the convergence of the second summand, we have used Lemma \ref{lem:bsl} and that $m/\gamma>2$. This proves the first assertion of the lemma. 
	
	Next, we prove the second claim. Choose $\tau \in (0, \min\{1,\gamma\})$. Using Chebyshev's inequality for the second inequality and (A4) in the third, we obtain
	\begin{align*}
	A&:=\limsup_{n \to \infty}\frac{n}{t_n}\Prob{ W^p > aq(n/t_n), |D_n-  \E_W D_n |>\varepsilon q(n/t_n)}\\
	&\leq\limsup_{n \to \infty}\frac{n}{t_n} \E \Big[\mathds{1}\{W^p > a q(n/t_n)\} \left(\mathbf{P}_{W}\left\{  |D_n - \E_W D_n| >\varepsilon q(n/t_n)\right\}\right)^\tau\Big]\\
	& \le \limsup_{n \to \infty}\frac{n}{t_n(\varepsilon q(n/t_n))^{2\tau}} \E \left[\mathds{1}\{W^p > aq(n/t_n)\} \left(\E_W| D_n - \E_W D_n|^2 \right)^{\tau}\right]\\
	&\le \limsup_{n \to \infty}\frac{n}{t_n(\varepsilon q(n/t_n))^{2\tau}} \E \left[\mathds{1}\{W^p > aq(n/t_n)\} a'_2 W^{p\tau}\right]
	\end{align*}
	for some constant $a'_2>0$. Note that
	\begin{align*}
	&\E \left[\mathds{1}\{W^p > aq(n/t_n)\} W^{p\tau}\right] \\
	& = \int_0^\infty \mathbf{P}\{\mathds{1}\{W^{p\tau} > (a q(n/t_n))^\tau\} W^{p\tau}>t\} dt \\
	& = (a q(n/t_n))^\tau \mathbf{P}\{W^{p\tau} > (a q(n/t_n))^\tau\} + \int_{(a q(n/t_n))^\tau}^\infty \mathbf{P}\{W^{p\tau}>t\} dt.
	\end{align*}
	Now by (A1), we have that $W^{p\tau}$ has a regularly varying tail with index $-\beta/(p\tau)<-1$. Hence, using Karamata's theorem in the first step and \eqref{eq:stail} in the second, we obtain 
	\begin{align*}
	A & \le \limsup_{n \to \infty}\frac{n}{t_n(\varepsilon q(n/t_n))^{2\tau}} \left[a'_2 \bigg( 1+ \frac{\tau}{\gamma-\tau}\bigg) (a q(n/t_n))^\tau \Prob{W^p>a q(n/t_n)}\right]\\
	& \le \limsup_{n \to \infty}\frac{n}{t_n\varepsilon^{2\tau} q(n/t_n)^{\tau}} \left[\frac{a'_2 \gamma}{\gamma-\tau} a^{\tau-\gamma}\frac{t_n}{n}\right]=0
	\end{align*}
	concluding the proof.
\end{proof}

We are now ready to prove Theorem~\ref{thm:exp}, which is the key intermediate step to prove the Theorems~\ref{thm.PoiLim} and \ref{thm:hillCons}.

\begin{proof}[Proof of Theorem~\ref{thm:exp}]
	It follows from (A2) that
	\begin{align*}
	A_1 :&=\frac{1}{t_n}\E \sum_{x \in \Delta_n} \mathds{1}\left\{D_{n,x}\ge q(n/t_n) \xi_n a \ge \xi_n W_x^p\right\}\\
	&= \frac{\E |\Delta_n|}{t_n} \Prob{D_n \ge q(n/t_n)\xi_n a \ge \xi_n W^p}.
	\end{align*}
	For ease of notation, we will let
	$$\overline D_n:=D_n-  \E_W D_n  \quad \text{and} \quad E_n:=\big|  \E_W D_n  -\xi_n W^p\big|.$$
	Fix $\varepsilon>0$. Note that
	\begin{align*}
	A_1  &\le \frac{\E|\Delta_n|}{t_n}\Big[\Prob{W^p \in q(n/t_n)[a-\varepsilon,a]}\\
	&\qquad +  \Prob{ W^p < q(n/t_n) (a-\varepsilon), |D_n-\xi_n W^p|>q(n/t_n)\xi_n \varepsilon}\Big]\\
	&\le \frac{\E|\Delta_n|}{t_n}\Big[\Prob{W^p \in q(n/t_n)[a-\varepsilon,a]}\\
	&\qquad +\Prob{ W^p < aq(n/t_n), |\overline D_n|>q(n/t_n)\xi_n \varepsilon - E_n }\Big]=:B_1+B_2.
	\end{align*}
	It follows from \eqref{eq:stail} and \eqref{eqn:Delta} that
	\begin{align*}
	\limsup_{n \to \infty} B_1&=\limsup_{n \to \infty} \frac{\E|\Delta_n|}{t_n}\Prob{W^p \in q(n/t_n)[a-\varepsilon,a]}\\
	&\le \limsup_{n \to \infty} \frac{\E|\Delta_n|}{t_n} (t_n/n)((a-\varepsilon)^{-\gamma}-a^{-\gamma})=((a-\varepsilon)^{-\gamma}-a^{-\gamma}),
	\end{align*}
	where the right-hand side goes to zero as $\varepsilon \to 0$.
	
	\noindent Next we consider $B_2$. Since we are interested in the behaviour as $n\to\infty$ and $q(n/t_n)\to \infty$ as $n\to\infty$, we assume in the following that $n$ is large enough so that
	\begin{equation*}
	q(n/t_n)\xi_n \varepsilon-E_n \geq q(n/t_n)\xi_n \varepsilon - C(aq(n/t_n))^{1-\varrho}\ge q(n/t_n)\xi \varepsilon/2>0,
	\end{equation*}
	where the first inequality holds almost surely by (A3) when $W^p < aq(n/t_n)$.
	Since $\E|\Delta_n|\sim n$ according to \eqref{eqn:Delta}, by Lemma~\ref{lem:part}, it now follows that $\limsup_{n \to \infty}B_2=0$. Hence, we obtain $\lim_{n \to \infty}A_1=0$ proving the first claim.
	
	Next we show that 
	\begin{equation*}
	A_2 := \frac{1}{t_n} \E \sum_{x \in \Delta_n} \mathds{1}\left\{\xi_n W_x^p\ge q(n/t_n)\xi_n a \ge D_{n,x}\right\}
	\end{equation*}
	also tends to 0 as $n \to \infty$. The proof follows a very similar line of argument as in the case of $A_1$. Fix $\varepsilon>0$ and note that for $n$ large enough, $A_2 \le B'_1 + B'_2$, where 
	$$
	B'_1 := \frac{\E|\Delta_n|}{t_n}\Prob{W^p \in q(n/t_n)[a,a+\varepsilon]}
	$$
	and
	$$
	B'_2 := \frac{\E|\Delta_n|}{t_n}\Prob{W^p > q(n/t_n)(a+\varepsilon), |\overline D_n|>q(n/t_n)\xi_n \varepsilon-E_{n}}.
	$$
	That $\lim_{\varepsilon \to 0} \limsup_{n \to \infty} B'_1=0$ follows again from \eqref{eq:stail} and \eqref{eqn:Delta}. For $B'_2$, we obtain
	\begin{align*}
	B'_2& \le \frac{\E|\Delta_n|}{t_n}\Big[\Prob{W^p > (aq(n/t_n))^{1+\varrho}}\\
	&\quad +\Prob{(aq(n/t_n))^{1+\varrho} \ge W^p > aq(n/t_n), |\overline D_n|>q(n/t_n)\xi_n \varepsilon-E_{n}}\Big]\\
	& =:	B'_{2,1} + B'_{2,2}
	\end{align*}
	with $\varrho$ as in (A3). Now, by \eqref{eq:stail} and \eqref{eqn:Delta}, we have that $B'_{2,1}$ goes to zero as $n \to \infty$. For $B'_{2,2}$, note that when $aq(n/t_n))^{1+\varrho} \ge W^p > aq(n/t_n)$, by (A3),
	$$
	q(n/t_n)\xi_n \varepsilon-E_{n} \ge q(n/t_n)\xi_n \varepsilon-C(aq(n/t_n))^{1-\varrho^2}\ge q(n/t_n)\xi\varepsilon/2
	$$
	for $n$ large enough. Hence, Lemma~\ref{lem:part} and \eqref{eqn:Delta} yield
	\begin{equation*}
	\limsup_{n \to \infty}B'_{2,2}
	\le \limsup_{n \to \infty} \frac{\E|\Delta_n|}{t_n}\Prob{W^p > aq(n/t_n), |\overline D_n|>q(n/t_n)\xi\varepsilon/2}=0.
	\end{equation*}
	Thus we obtain $\lim_{n \to \infty} A_2=0$, which concludes the proof.
\end{proof}

\subsection{Proof of Theorem~\ref{thm.PoiLim}}

In this section, we prove Theorem~\ref{thm.PoiLim}, together with the following result, which is an important ingredient for the proof of Theorem~\ref{thm:hillCons}. For $k,n \in\N$ let $\mathcal{D}_{k,n}$ be the random measure
$$
\mathcal{D}_{k,n}=\frac{1}{k} \sum_{x \in \Delta_n} \delta_{\frac{D_{n,x}}{\xi_n q(n/k)}}.
$$

\begin{prop}\label{prop.LimMeasure} Assume that (A1)-(A4) are satisfied. Then for an intermediate sequence $(k_n)_{n \in \N}$, as $n \to \infty$,
	$$
	\mathcal{D}_{k_n,n} \xrightarrow{d} \nu_{\gamma} \quad \text{in }  M_+((0,\infty]).
	$$
\end{prop}	

Similarly as in Theorem~\ref{thm.PoiLim}, we consider the convergence in $M_+((0,\infty])$ since the measure $\nu_\gamma$ is not locally finite at zero.

As a final ingredient for the proof of Theorem~\ref{thm.PoiLim}, we provide Lemma \ref{lem:welim} below, which proves that the transformed weight sequence $(q(n)^{-1}W_x^p)_{x \in \Delta_n}$ converges to a Poisson process as $n \to \infty$, while scaling by $q(n/k_n)$ results in a deterministic measure as limit.

For $k,n \in \N$, define
$$
\mathcal{E}_n=\sum_{x \in \Delta_n} \delta_{\frac{W_x^p}{q(n)}} \quad \text{and} \quad \mathcal{E}_{k,n}=\frac{1}{k} \sum_{x \in \Delta_n} \delta_{\frac{ W_x^p}{q(n/k)}}.
$$
Let $\hat K$ denote the class of relatively compact sets in $(0,\infty]$. Note that the collection of sets $\mathcal{I} \subset \hat K$ given by
\begin{equation}\label{eq:I}
\mathcal{I}=\{(a,b]: 0 < a\le b\le \infty\}
\end{equation}
forms a \textit{dissecting semi-ring}, see e.g.\ \cite{K17}, while the class $\mathcal{U}$ of finite unions of sets from $\mathcal{I}$ forms a \textit{dissecting ring}.
\begin{lemma}\label{lem:welim}
	Let $(k_n)_{n\in \N}$ be an intermediate sequence and let $\eta_\gamma$ be a Poisson process with intensity measure $\nu_\gamma$. Then as $n \to \infty$,
	$$
	\mathcal{E}_n \xrightarrow{d} \eta_{\gamma} \quad \text{in } M_p((0,\infty]) \quad \text{and} \quad \mathcal{E}_{k_n,n} \xrightarrow{d} \nu_{\gamma}\quad \text{in } M_+((0,\infty]).
	$$
\end{lemma}
\begin{proof} Let $\mathcal{I}$ and $\mathcal{U}$ be as above. By \cite[Theorem 4.15]{K17}, the first claim follows if we show that
	\begin{enumerate}[(a)]
		\item $\lim_{n \to \infty}\Prob{ \mathcal{E}_n (A) = 0} =
		\Prob{ \eta_{\gamma} (A) = 0}$ for all $A \in \mathcal{U}$,
		\item $\limsup_{n \to \infty} \Prob{\mathcal{E}_n (B) > 1 } \leq
		\Prob { \eta_{\gamma} (B )>1} $ for all $B \in \mathcal{I}$.
	\end{enumerate}
	To prove (a), fix $A \in \mathcal{U}$ and note that we can write $A=\bigcup_{i=1}^k(a_i,b_i]$ for $k \in \N$, and disjoint intervals $(a_i,b_i]$ with $0<a_i \le b_i \le \infty$ for $1 \le i \le k$. Since given $|\Delta_n|$, the weights $(W_x)_{x \in \Delta_n}$ are independent and distributed as $F(\cdot)$, letting $p_{n,i}=\Prob{W^p \in q(n)(a_i,b_i]}$, we have
	$$
	\Prob{ \mathcal{E}_n (A) = 0}=\E\left[\left(1-\sum_{i=1}^k p_{n,i}\right)^{|\Delta_n|}\right]=\E\left[x_n^{|\Delta_n|/n}\right],
	$$
	where
	$x_n=\left(1-\sum_{i=1}^kp_{n,i}\right)^{n}$. Note that by \eqref{eq:stail},
	$$
	\lim_{n \to \infty}x_n=\exp\left(-\sum_{i=1}^k(a_i^{-\gamma}-b_i^{-\gamma})\right)=\Prob{ \eta_{\gamma} (A) = 0}.
	$$
	Together with $|\Delta_n|/n\xrightarrow{\mathbf{P}}1$ as $n\to\infty$ by \eqref{eqn:Delta}, we see that $x_n^{|\Delta_n|/n}\xrightarrow{\mathbf{P}} \Prob{ \eta_{\gamma} (A) = 0}$ as $n\to\infty$. Since the sequence $(x_n^{|\Delta_n|/n})_{n\in\N}$ is uniformly bounded by one, this yields 
	\begin{equation}\label{eq:(a)}
	\lim_{n \to \infty}\Prob{ \mathcal{E}_n (A) = 0}=\lim_{n \to \infty}\E\left[x_n^{|\Delta_n|/n}\right]=\Prob{ \eta_{\gamma} (A) = 0}
	\end{equation}
	proving (a).
	
	\noindent Next we show (b). Fix $0<a \le b \le \infty$ and let $p_n=\Prob{W^p \in q(n)(a,b]}$. One has
	\begin{align*}
	\Prob{ \mathcal{E}_n ((a,b]) = 1}&=\E\left[\frac{p_n}{1-p_n}|\Delta_n|\left(1-p_n\right)^{|\Delta_n|}\right]. 
	\end{align*}
	Now from \eqref{eq:stail}, we have 
	$$
	\lim_{n \to \infty} \frac{n p_n}{1-p_n}=a^{-\gamma}-b^{-\gamma} \quad \text{and} \quad \lim_{n \to \infty}(1-p_n)^n=e^{-(a^{-\gamma}-b^{-\gamma})}.
	$$
	Note that $|\Delta_n| p_n (1-p_n)^{|\Delta_n| -1} \le 1$ since the left-hand side is the probability that a $\mathrm{Binomial}(|\Delta_n|,p_n)$ random variable takes the value one. Now 
	arguing similarly as for (a) yields
	$$
	\lim_{n \to \infty} \Prob{ \mathcal{E}_n ((a,b]) = 1} = (a^{-\gamma}-b^{-\gamma})e^{-(a^{-\gamma}-b^{-\gamma})}.
	$$
	This along with \eqref{eq:(a)} implies (b), yielding the first assertion of the lemma.
	
	Next we prove the second claim of the lemma. By \cite[Theorem 1.1]{K73}, it is enough to show that for all $B_1,\dots,B_k \in \mathcal{I}$, $k \in \N$,
	$$\left(\mathcal{E}_{k_n,n}(B_1), \dots,\mathcal{E}_{k_n,n}(B_k)\right) \xrightarrow{\mathbf{P}} \left(\nu_{\gamma} (B_1), \dots, \nu_{\gamma} (B_k)\right) \quad \text{as $n \to \infty$},$$
	which in turn follows from the one dimensional convergence
	\begin{equation}\label{eq:EnOne}
	\mathcal{E}_{k_n,n}(B) \xrightarrow{\mathbf{P}} \nu_{\gamma} (B) \quad \text{ as $n \to \infty$}
	\end{equation}
	for any $B \in \mathcal{I}$. Fix $0<a \le b \le \infty$. First note that since $|\Delta_n|/n \xrightarrow{\mathbf{P}} 1$ as $n \to \infty$, by \eqref{eq:stail} and Slutsky's theorem, we have that as $n \to \infty$,
	\begin{equation}\label{eq:Tn}
	T_{n}:=\frac{|\Delta_n|}{k_n} \Prob{W^p \in q(n/k_n)(a,b]} \xrightarrow{\mathbf{P}}a^{-\gamma}-b^{-\gamma}=\nu_\gamma((a,b]).
	\end{equation}
	For $\eps>0$, again using that $|\Delta_n|/n \xrightarrow{\mathbf{P}} 1$ as $n \to \infty$, we can write
	\begin{align*}
	&\limsup_{n \to \infty}\mathbf{P}\left\{\Big| \mathcal{E}_{k_n,n}((a,b]) - T_n \Big|> \eps \right\}\\
	&\le \limsup_{n \to \infty}\E \left[\mathds{1}\left\{\Big| \frac{|\Delta_n|}{n} -1\Big| \le 1\right\} \mathbf{P}\left\{\big|\mathcal{E}_{k_n,n}((a,b]) - T_n \big|>\eps \Big| |\Delta_n|\right\}\right].
	\end{align*}
	Since $\E \left[\mathcal{E}_{k_n,n}((a,b]) \Big| |\Delta_n|\right]=T_n$, using Chebyshev's inequality, we obtain
	\begin{align*}
	&\limsup_{n \to \infty}\mathbf{P}\left\{\Big| \mathcal{E}_{k_n,n}((a,b]) - T_n \Big|> \eps \right\}\\
	&\le  \limsup_{n \to \infty} \eps^{-2}\E \left[\mathds{1}\left\{\Big| \frac{|\Delta_n|}{n} -1\Big| \le 1\right\} {\rm Var}\left(\mathcal{E}_{k_n,n}((a,b]) \Big| |\Delta_n|\right)\right]\\
	&\le  \limsup_{n \to \infty}\frac{2n}{\eps^2 k_n^2} \Prob{W^p \in q(n/k_n)(a,b]} =0,
	\end{align*}
	where in the last step, we have used \eqref{eq:stail}. Since $\eps>0$ was arbitrary, this and \eqref{eq:Tn} imply \eqref{eq:EnOne} by Slutsky's theorem, thus concluding the proof.
\end{proof}
We are now ready to prove Theorem~\ref{thm.PoiLim} and Proposition~\ref{prop.LimMeasure}.
\begin{proof}[Proof of Theorem~\ref{thm.PoiLim} and Proposition~\ref{prop.LimMeasure}]
	Let $\mathcal{I}$ be as in \eqref{eq:I}. By \cite[Theorem 1.1]{K73}, the first assertion in Theorem~\ref{thm.PoiLim} follows if we show that for all $B_1,\dots,B_k \in \mathcal{I}$, $k \in \N$,
	\begin{equation*}
	\left(\mathcal{D}_n(B_1), \dots,\mathcal{D}_n(B_k)\right) \xrightarrow{d} \left(\eta_{\gamma} (B_1), \dots, \eta_{\gamma} (B_k)\right) \quad \text{as $n \to \infty$.}
	\end{equation*}
	Since by Lemma~\ref{lem:welim} we have that as $n \to \infty$,
	\begin{equation*}
	\left(\mathcal{E}_n(B_1), \dots,\mathcal{E}_n(B_k)\right) \xrightarrow{d} \left(\eta_{\gamma} (B_1), \dots, \eta_{\gamma}(B_k)\right),
	\end{equation*}
	it is enough to show that for all $B_1,\dots,B_k \in \mathcal{I}$, $k \in \N$, as $n \to \infty$,
	\begin{equation*}
	\left(\mathcal{D}_n(B_1), \dots,\mathcal{D}_n(B_k)\right) - 	\left(\mathcal{E}_n(B_1), \dots,\mathcal{E}_n(B_k)\right) \xrightarrow{\mathbf{P}} 0,
	\end{equation*}
	which in turn follows if we show that for $0<a \le b\le \infty$,
	\begin{equation}\label{eq:convd}
	\mathcal{D}_{n} ((a,b]) -\mathcal{E}_{n} ((a,b])\xrightarrow{\mathbf{P}} 0 \quad \text{as $n \to \infty$.}
	\end{equation}
	Fix $0<a<\infty$. Note that
	\begin{align}\label{eq:expdiff}
	\E &\big|\mathcal{D}_{n} ((a,\infty]) -\mathcal{E}_{n} ((a,\infty])\big| \nonumber\\
	\le& \E \sum_{x \in \Delta_n} \mathds{1}\{\xi_n^{-1}D_{n,x}\ge q(n) a \ge W_x^p\}+\mathds{1}\{\xi_n^{-1}D_{n,x} \le q(n) a \le  W_x^p\}.
	\end{align}
	Thus, by Theorem \ref{thm:exp} we have that as $n \to \infty$,
	\begin{equation*}
	\mathcal{D}_{n} ((a,\infty])-\mathcal{E}_{n} ((a,\infty]) \xrightarrow{\mathbf{P}} 0.
	\end{equation*}
	For $0<a \le b<\infty$, as
	$$
	\mathcal{D}_{n} ((a,b])-\mathcal{E}_{n} ((a,b])=\left[\mathcal{D}_{n} ((a,\infty])-\mathcal{E}_{n} ((a,\infty])\right] - \left[\mathcal{D}_{n} ((b,\infty])-\mathcal{E}_{n} ((b,\infty])\right],
	$$ 
	it follows that as $n \to \infty$,
	\begin{equation*}
	\mathcal{D}_{n} ((a,b])-\mathcal{E}_{n} ((a,b]) \xrightarrow{\mathbf{P}} 0.
	\end{equation*}
	This shows \eqref{eq:convd} proving \eqref{ProConv}.
	
	\noindent Next, for $a > 0$, using \eqref{ProConv}, we obtain
	\begin{align*}
	\lim_{n \to \infty}\Prob{\xi_n^{-1}q(n)^{-1}\max_{x\in \Delta_n}D_{n,x} \le a}&=\lim_{n \to \infty} \Prob{\mathcal{D}_n ((a,\infty])=0}\\
	&=\Prob{\eta_{\gamma}((a,\infty])=0}=\exp\left(- a^{-\gamma}\right),
	\end{align*}
	which provides the Frechet limit in \eqref{eq:Frl} concluding the proof of Theorem~\ref{thm.PoiLim}.	
	
	Finally we prove Proposition~\ref{prop.LimMeasure}. For an intermediate sequence $(k_n)_{n \in \N}$, by \cite[Theorem 1.1]{K73} and Lemma \ref{lem:welim}, arguing as above, we see that it is enough to show that for $0<a \le b\le \infty$, as $n \to \infty$,
	\begin{equation*}
	\mathcal{D}_{k_n,n} ((a,b])-\mathcal{E}_{k_n,n} ((a,b]) \xrightarrow{\mathbf{P}} 0.
	\end{equation*}
	This follows from Theorem \ref{thm:exp} by arguing as in \eqref{eq:expdiff} and concludes the proof of Proposition~\ref{prop.LimMeasure}.
\end{proof}

Before moving on to prove Theorem~\ref{thm:hillCons}, we provide the following theorem which shows that for any $k \in \N$, with high probability, the vertex with the $k$-th largest weight has the $k$-th largest degree as $n\to\infty$. The result is a consequence of Theorem~\ref{thm.PoiLim}.

\begin{theorem}\label{thm:1-1}
	For $n, k \in \N$ let 
	\begin{align*}
	\mathcal{A}_{k,n} = \big\{ & \text{the vertex in $\Delta_n$ with the $k$-th largest} \\
	& \text{weight has the $k$-th largest degree}   \big\} \cap \big\{ k\leq |\Delta_n| \big\}.
	\end{align*}
	Then
	$$
	\lim_{n \to \infty}\Prob{\mathcal{A}_{k,n}}=1.
	$$
\end{theorem}

\begin{proof}
	Fix $a>0$ and note that
	\begin{align*}
	\Prob{\mathcal{A}_{k,n}^c} &\le \Prob{W_x <W_y, D_{n,x} \ge D_{n,y}, D_{n,x} >a \xi_nq(n) \text{ for some } x,y \in \Delta_n}\\
	&\quad +\Prob{D_{(k)}(n) \le a \xi_n q(n)} =: A+B,
	\end{align*}
	where as a convention, we let $D_{(k)}(n)=-\infty$ if $k >|\Delta_n|$.
	
	Fix $\delta>0$ small and partition $(a,\infty]$ as $(a_0,a_1, a_2, \dots, a_m,a_{m+1})$ with $a_0=a$ and $a_{m+1}=\infty$ for a fixed integer $m \in \N$ so that $\nu_\gamma((a_i,a_{i+1}]) \le \delta$ for all $i=0, \dots, m$. This leads to
	\begin{align*}
	&A\le \sum_{i=0}^m \Prob{\mathcal{D}_n((a_i,a_{i+1}]) \ge 2} + 
	\sum_{i=0}^{m}\E \sum_{x \in \Delta_n} \mathds{1}\{\xi_n W_x^p\le q(n)\xi_n a_i<D_{n,x}\}\\
	&\qquad \qquad + \sum_{i=0}^{m} \E \sum_{x \in \Delta_n}\mathds{1}\{\xi_n W_x^p>q(n) \xi_n a_i\ge D_{n,x}\}.
	\end{align*}
	Since $m$ is fixed, by Theorem~\ref{thm:exp}, the second and the third summand go to zero as $n \to \infty$. On the other hand, by \eqref{ProConv}, we obtain
	\begin{align*}
	&\lim_{n \to \infty} \sum_{i=0}^m \Prob{\mathcal{D}_n((a_i,a_{i+1}]) \ge 2}=\sum_{i=0}^m \Prob{\eta_{\gamma}((a_i,a_{i+1}]) \ge 2} \\
	&= \sum_{i=0}^m e^{-\nu_{\gamma}((a_i,a_{i+1}])}\sum_{j=2}^\infty \frac{\nu_{\gamma}((a_i,a_{i+1}])^j}{j!} \\
	& \le \sum_{i=0}^m \nu_{\gamma}((a_i,a_{i+1}])^2 \le \delta \nu_{\gamma}((a,\infty])=\delta a^{-\gamma},
	\end{align*}
	whence
	$$
	\limsup_{n \to \infty}A \le \delta a^{-\gamma}.
	$$
	Next, note that
	\begin{align*}
	\lim_{n \to \infty} B = \lim_{n \to \infty}\Prob{\mathcal{D}_n((a,\infty]) \le k -1} & =\Prob{\eta_{\gamma}((a,\infty]) \le k -1}\\
	& = e^{-a^{-\gamma}}  \sum_{i=0}^{k-1} \frac{a^{-i \gamma }}{i!}.
	\end{align*}
	Thus we obtain
	$$
	\limsup_{n \to \infty}\Prob{\mathcal{A}_{k,n}^c} \le \delta a^{-\gamma} + e^{-a^{-\gamma}}  \sum_{i=0}^{k-1} \frac{a^{-i \gamma }}{i!}.
	$$
	Since $\delta$ and $a$ were chosen arbitrarily, first taking $\delta \to 0$ and then taking $a \to 0$, we obtain
	$$\limsup_{n \to \infty}\Prob{\mathcal{A}_{k,n}^c}=0$$
	concluding the proof.
\end{proof}

\subsection{Proof of Theorem~\ref{thm:hillCons}}\label{sec:Hill}
In this section, we establish our second main result Theorem~\ref{thm:hillCons}. We first prove Lemma~\ref{lem:hill} below, which shows that for an intermediate sequence $(k_n)_{n \in \N}$, the random variable $D_{(k_n+1)}(n)$ is comparable to $\xi_n q(n/k_n)$ when $n$ is large. We note here that our proof very loosely adapts the proof of \cite[Theorem 10]{RW19}, but the arguments differ significantly. In particular, a parallel of Lemma~\ref{lem:hill} which is a key step in the proof of Theorem~\ref{thm:hillCons}, is proved there adapting certain techniques from \cite{AGS08} to embed the degree sequence of a linear preferential attachment model into a process constructed from a birth process with immigration. Instead, we make use of the convergence in Proposition~\ref{prop.LimMeasure}.

\begin{lemma}\label{lem:hill} Let $(k_n)_{n \in \N}$ be an intermediate sequence and assume that, for all $t>0$, $\mathcal{D}_{k_n,n}((t,\infty]) \xrightarrow{\mathbf{P}} \nu_{\gamma}((t,\infty])$ as $n \to \infty$. Then,
	$$\frac{D_{(k_n+1)}(n)}{\xi_n q(n/k_n)} \xrightarrow{\mathbf{P}} 1 \quad \text{as $n \to \infty$}.$$
\end{lemma}
\begin{proof}
	For $t>0$ let $G_n(t)=\mathcal{D}_{k_n,n}((t,\infty])$ and $G(t)=\nu_\gamma((t,\infty])=t^{-\gamma}$. By our assumption, we have that for all $t>0$,
	\begin{equation}\label{eq:c}
	G_n(t) \xrightarrow{\mathbf{P}}G(t) \quad \text{as }\; n \to \infty.
	\end{equation}
	We will prove the result by contradiction. Fix $\varepsilon>0$ and let $\mathcal{A}_n$ and $\mathcal{B}_n$ denote the events
	$$\mathcal{A}_n=\left\{\frac{D_{(k_n+1)}(n)}{\xi_n q(n/k_n)}>1+\varepsilon\right\} \quad \text{and} \quad \mathcal{B}_n=\left\{\frac{D_{(k_n+1)}(n)}{\xi_n q(n/k_n)}<1-\varepsilon\right\}.$$
	Assume that $\limsup_{n \to \infty} \Prob{\mathcal{A}_n} \ge \delta>0$ for some $\delta>0$. Since under $\mathcal{A}_n$, there are at least $k_n+1$ points of $\mathcal{D}_{k_n,n}$ in the interval $(1+\varepsilon, \infty]$, one has that $G_n(1+\varepsilon) \ge (k_n+1)/k_n>1$ on $\mathcal{A}_n$ while $G(1+\varepsilon)=(1+\varepsilon)^{-\gamma}<1$. Hence, we obtain
	$$
	\limsup_{n \to \infty}\Prob{|G_n(1+\varepsilon)-G(1+\varepsilon)|>1-(1+\varepsilon)^{-\gamma}} \ge \limsup_{n \to \infty}\Prob{\mathcal{A}_n}\ge \delta>0
	$$
	contradicting \eqref{eq:c}. Thus, we obtain that $\lim_{n \to \infty}\Prob{\mathcal{A}_n}=0$. Noting that under $\mathcal{B}_n$, one has that $G_n(1-\varepsilon) \le 1$, while $G(1-\varepsilon)=(1-\varepsilon)^{-\gamma}>1$, a similar argument as above shows that 
	$\lim_{n \to \infty}\Prob{\mathcal{B}_n}=0$. Since $\varepsilon>0$ was arbitrarily chosen, this proves the result.
\end{proof}
Before proving Theorem~\ref{thm:hillCons}, we need to establish another lemma.
\begin{lemma}\label{lem:contmap}
	Assume that, for all $t>0$, $\mathcal{D}_{k_n,n}((t,\infty]) \xrightarrow{\mathbf{P}} \nu_{\gamma}((t,\infty])$ as $n \to \infty$. Then for $0<a<b<\infty$, as $n \to \infty$,
	$$	\int_{a}^{b}\mathcal{D}_{k_n,n}((y,\infty])\frac{dy}{y}\xrightarrow {\mathbf{P}}\int_{a}^{b}\nu_{\gamma}((y,\infty])\frac{ dy}{y}.$$
\end{lemma}

\begin{proof}
	For $k \in \N$ and $y_i=a+(b-a)i/k$ for $0 \le i \le k$, we have that
	$$
	\frac{b-a}{k}\sum_{i=1}^k \frac{\mathcal{D}_{k_n,n}((y_{i},\infty])}{y_{i}} \le \int_{a}^{b}\mathcal{D}_{k_n,n}((y,\infty])\frac{dy}{y} \le \frac{b-a}{k}\sum_{i=1}^k \frac{\mathcal{D}_{k_n,n}((y_{i-1},\infty])}{y_{i-1}}.
	$$
	Now, by our assumption and Slutsky's theorem, the lower and the upper bounds converge in probability to $$\Sigma_l:=\frac{b-a}{k}\sum_{i=1}^k \frac{\nu_\gamma((y_{i},\infty])}{y_{i}}\quad \text{and} \quad \Sigma_u:=\frac{b-a}{k}\sum_{i=1}^k \frac{\nu_\gamma((y_{i-1},\infty])}{y_{i-1}},$$
	respectively, as $n \to \infty$. 
	
	Fix $\varepsilon>0$. Since both $\Sigma_l$ and $\Sigma_u$ converge to $\int_{a}^{b}\nu_{\gamma}((y,\infty])\frac{ dy}{y}$ as $k \to \infty$, choose $k$ large enough so that $|\Sigma_l-\Sigma_u|<\varepsilon/2$. As 
	$$
	\Sigma_l \le \int_{a}^{b}\nu_{\gamma}((y,\infty])\frac{ dy}{y} \le \Sigma_u,$$
	for $k\in \N$ chosen as above, one has
	\begin{align*}
	&\mathbf{P}\left\{\Bigg| \int_{a}^{b}\mathcal{D}_{k_n,n}((y,\infty])\frac{dy}{y}-\int_{a}^{b}\nu_{\gamma}((y,\infty])\frac{ dy}{y}\Bigg|>\varepsilon\right\}\\
	&\le \mathbf{P}\left\{\Bigg| \frac{b-a}{k}\sum_{i=1}^k \frac{\mathcal{D}_{k_n,n}((y_{i},\infty])}{y_{i}}-\Sigma_l\Bigg|>\varepsilon/2 \right\}\\
	&\qquad\qquad\qquad\qquad +\mathbf{P}\left\{\Bigg|\frac{b-a}{k}\sum_{i=1}^k \frac{\mathcal{D}_{k_n,n}((y_{i-1},\infty])}{y_{i-1}}-\Sigma_u\Bigg|>\varepsilon/2 \right\}.
	\end{align*}
	As noted above, both summands in the upper bound go to zero as $n \to \infty$. Noting that the choice of $\varepsilon>0$ was arbitrary, this concludes the proof.
\end{proof}
Finally, we prove Theorem~\ref{thm:hillCons}, which shows the consistency of the Hill estimator.

\begin{proof}[Proof of Theorem~\ref{thm:hillCons}:]
	By Proposition \ref{prop.LimMeasure}, the assumption that, for all $t>0$, $\mathcal{D}_{k_n,n}((t,\infty]) \xrightarrow{\mathbf{P}} \nu_{\gamma}((t,\infty])$ as $n \to \infty$ is satisfied in Lemma \ref{lem:hill} and in Lemma \ref{lem:contmap}. Note that
	$$
	H_{k_{n},n}=\int_{\frac{D_{(k_n+1)}(n)}{\xi_n q(n/k_n)}}^{\infty}\mathcal{D}_{k_n,n}((y ,\infty])\frac {dy}{y}.
	$$
	For $M>1$, by Lemma~\ref{lem:contmap}, we have that as $n \to \infty$,
	$$
	\int_{1}^{M}\mathcal{D}_{k_n,n}((y,\infty])\frac{dy}{y}\xrightarrow {\mathbf{P}}\int_{1}^{M}\nu_{\gamma}((y,\infty])\frac{ dy}{y} = \frac{1}{\gamma} (1-M^{-\gamma}).
	$$
	Next, we claim that as $n \to \infty$,
	\begin{equation}\label{eq:1app}
	\int_{1}^{M}\mathcal{D}_{k_n,n}((y,\infty])\frac{dy}{y}-\int_{\frac{D_{(k_n+1)}(n)}{\xi_n q(n/k_n)}}^{M}\mathcal{D}_{k_n,n}((y ,\infty])\frac {dy}{y} \xrightarrow{\mathbf{P}}0.
	\end{equation}
	To see this, fix $\varepsilon>0$ and note that for $\delta \in(0,1)$,
	\begin{align*}
	&\mathbf{P}\left\{\int_{\min\left\{1, \frac{D_{(k_n+1)}(n)}{\xi_n q(n/k_n)}\right\}}^{\max\left\{1, \frac{D_{(k_n+1)}(n)}{\xi_n q(n/k_n)}\right\}}\mathcal{D}_{k_n,n}((y ,\infty])\frac{dy}{y} >\varepsilon\right\}\\
	&\le  \mathbf{P}\left\{\int_{1-\delta}^{1+\delta}\mathcal{D}_{k_n,n}((y ,\infty])\frac{dy}{y} >\varepsilon\right\}+\mathbf{P}\left\{\Bigg|\frac{D_{(k_n+1)}(n)}{\xi_n q(n/k_n)}-1\Bigg|>\delta \right\}  =:  A+B.
	\end{align*}
	Again by Lemma~\ref{lem:contmap} we have,
	$$
	\limsup_{n \to \infty} A \le  \mathds{1}\left\{\int_{1-\delta}^{1+\delta}\nu_{\gamma}((y,\infty])\frac{ dy}{y} \ge  \varepsilon\right\}=\mathds{1}\left\{ (1-\delta)^{-\gamma}-(1+\delta)^{-\gamma} \geq \gamma\varepsilon\right\},
	$$
	which can be made equal to zero by choosing $\delta>0$ small enough. On the other hand, it follows from Lemma~\ref{lem:hill} that $B \to 0$ as $n \to \infty$, which proves \eqref{eq:1app}. By \cite[Theorem 4.28]{K06}, it is now enough to show that
	\begin{equation}\label{eq:Go'}
	\lim_{ M \rightarrow \infty} \limsup_{ n \rightarrow \infty} \E \left[\int _ { M } ^ { \infty } \mathcal{D}_{k_n,n}((y ,\infty]) \frac {dy } { y }\right] = 0.
	\end{equation}
	Using (A2) for the first equality and substituting $\xi_n y q(n/k_n)$ by $z$ in the second step, we can rewrite the expectation as
	\begin{align*}
	\E \left[\int _ { M } ^ { \infty } \mathcal{D}_{k_n,n}((y ,\infty]) \frac {dy}{y }\right] &=\frac{\E |\Delta_n|}{k_n} \int_{M}^\infty \Prob{D_n>\xi_n y q(n/k_n)}\frac {dy}{y}\\ &=\frac{\E |\Delta_n|}{k_n} \int_{M\xi_n q(n/k_n)}^\infty \Prob{D_n>z}\frac {dz}{z}.
	\end{align*}
	Since $z\mapsto \Prob{D_n>z}$ belongs to $\mathrm{RV}_{-\gamma}$ by Theorem~\ref{thm:deg}, the function $z\mapsto \Prob{D_n>z}/z$ belongs to $\mathrm{RV}_{-\gamma-1}$. Thus, noting that $-\gamma-1<-1$, it follows from Karamata's theorem (see e.g.\ \cite[Theorem 2.1]{Re07}) that as $n \to \infty$,
	$$
	\int_{M\xi_n q(n/k_n)}^\infty \Prob{D_n>z}\frac {dz}{z}\sim \frac{1}{\gamma } \Prob{D_n>M\xi_n q(n/k_n)}.
	$$
	Thus, we obtain
	\begin{align}\label{eq:ub}
	&\limsup_{n \to \infty}\E \left[\int _{M}^{\infty} \mathcal{D}_{k_n,n}((y ,\infty]) \frac {dy } {y}\right]=\limsup_{n \to \infty} \frac{\E |\Delta_n|}{\gamma k_n} \Prob{D_n>M\xi_n q(n/k_n)}\nonumber\\
	&\le \limsup_{n \to \infty}\frac{\E |\Delta_n|}{\gamma k_n} \Prob{W^p>M q(n/k_n)}\nonumber\\
	&\qquad \qquad \quad  +\limsup_{n \to \infty}\frac{\E |\Delta_n|}{\gamma k_n} \Prob{D_n>M \xi_n q(n/k_n) \ge \xi_n W^p}.
	\end{align}
	By (A2) and Theorem \ref{thm:exp}, the second term on the right hand side of \eqref{eq:ub} is equal to zero and by \eqref{eq:stail} and \eqref{eqn:Delta}, the first term is equal to $M^{-\gamma}/\gamma$, which tends to zero as $M \to \infty$. This yields \eqref{eq:Go'}, thus proving the result.
\end{proof}

\section{Applications}\label{sec:app}
In this section, we consider the random graph models {\bf I}--{\bf V} considered in Section~\ref{sec:2} and prove Theorem~\ref{thm:I-V}. For the models {\bf I}--{\bf IV} the assumptions (A2)-(A4) follow from only assuming (A1) and the construction of the graphs. We show this in Section~\ref{subsec1} proving Theorem \ref{thm:I-V} (a)-(c). In Section~\ref{subsec2}, we relate the model {\bf V} to the model {\bf IV} by a coupling argument, which yields Theorem~\ref{thm:I-V} (d).

\subsection{Checking the assumptions for models I-IV}\label{subsec1}
We start by showing that (A2) holds for these models. First note that if the vertex sets $(V_n)_{n \in \N}$ are deterministic with 
\begin{equation}\label{eq:DegWe}
(D_{n,x},W_x)=_d (D_{n,y},W_y) \quad \text{for all $x,y \in V_n$, $n \in \N$,}
\end{equation}
then one can take the degree and the weight of any fixed vertex $x_n \in V_n$ as the typical degree and weight respectively as it is straightforward to see from \eqref{eq:DegWe} that they satisfy (A2). The graphs {\bf I}, {\bf II} and {\bf IV} all satisfy \eqref{eq:DegWe} by symmetry. In the case of {\bf III}, since it has a Poisson process as its vertex set, specifying the typical degree-weight pair requires some justification. As mentioned in Section~\ref{sec:2}, for a Poisson process, its Palm measure is the measure when one adds the origin $0$ to the original Poisson process. Extending the graph to include $0$ as a vertex, it is not hard to see that the degree and the weight of the vertex $0$ satisfy (A2).

Next we check (A3). One typically needs some natural conditions on the model parameters for (A3) to hold. Interestingly, proving (A3) is a crucial step in the investigations of these models in \cite{DHH13,Y06,DW18,NR06}. The following lemma combines several results from \cite{DHH13,Y06,DW18,NR06} with small modifications providing sufficient conditions for (A3) to hold in the graphs {\bf I}--{\bf IV}.
\begin{lemma}\label{lem:A3}
	Assume that (A1) is satisfied. Then the following are true.
	\begin{enumerate}[(a)]
		\item For models {\bf I} and {\bf III}, if $d<\min\{\alpha,\alpha\beta\}$, then (A3) holds for some universal constant $C$ ($C\le v_d$ for {\bf I} while $C=0$ for {\bf III}) with $p=d/\alpha$, $\varrho=1$ and for all $n \in \N$,
		$$\xi_n=\xi=\lambda^{d/\alpha}v_{d}\Gamma\left(1-\frac{d}{\alpha}\right)\E\left[W^{d/\alpha}\right],$$
		where $v_d$ denotes the volume of the unit ball in $\R^d$ and $\Gamma(\cdot)$ is the Gamma function.
		\item For model {\bf II}, if $\beta<d$, then (A3) holds for some universal constant $C$ with $p=d -\beta>0$, $\varrho=\min\{1,(d-\beta)^{-1}\}$ and $\xi_n=\xi=dv_d/(d-\beta)$, $n \in \N$.
		\item For model {\bf IV}, (A3) holds with $p=\xi_n=\xi=\varrho=1$ for all $n \in \N$ and $C=0$.
	\end{enumerate}
\end{lemma}
\begin{proof}
	For models {\bf I} and {\bf III}, the result follows from \cite[Proposition 2.3]{DHH13} and \cite[Lemma 4.1 and Proposition 4.3]{DW18}, respectively, while for models {\bf II} and {\bf IV}, it is a consequence of \cite[Lemma 2.1]{Y06} and \cite[Proposition 2.1(ii)]{NR06}, respectively.	
\end{proof}

Finally we check that (A4) holds for the models {\bf I}--{\bf IV} when (A1) is satisfied. 
Note that for models {\bf III} and {\bf IV}, by \cite[Lemma 4.1 and Proposition 4.3]{DW18} and \cite[Proposition 2.1 (ii)]{NR06}, respectively, given $W$, we have that $D_n$ is Poisson distributed with mean $\xi W^p$, where $\xi$ and $p$ are as in Lemma~\ref{lem:A3}. Hence, by properties of the Poisson distribution, (A4) holds trivially for these models.

For the models {\bf I} and {\bf II}, the typical degree $D_n$ and the typical weight $W$ can be chosen as the degree $D_{n,0}$ of the origin and the associated weight $W_0$. Writing $D_{n,0}$ as
$$D_{n,0} = \sum_{x \in V_n} \mathds{1}_{x,n},$$
where $\mathds{1}_{x,n} :=\mathds{1}\left\{\{0,x\}  \text{ is an edge of $G_n$} \right\}$, notice that the random variables $(\mathds{1}_{x,n})_{x \in V_n}$ are conditionally independent given $W_0$. In this special case, the following lemma shows that (A4) always holds. 

\begin{lemma}\label{lem:momg}
	Assume that (A2)-(A3) are in force, that the vertex sets $(V_n)_{n \in \N}$ are deterministic and that $(D_n,W) = \big(\sum_{x \in V_n} \mathds{1}_{x,n}, W_0\big)$, where $(\mathds{1}_{x,n})_{x \in V_n}$ are conditionally independent $\{0,1\}$-valued random variables given $W_0$. Then (A4) holds.
\end{lemma}
\begin{proof}
	We first prove the assertion for all even $m\ge 2$. Let $m=2k$ for some $k \in\N$. Using the conditional independence of the indicators in the first step and (A3) in the second, we have that $\mathbf{P}$-a.s.
	\begin{equation}\label{eq:Var}
	\E_{W} (D_n-\E_W D_n)^2  \le \E_W D_n \le \xi_n W^p + C\max\left\{1,W^{p(1-\varrho)}\right\} \le \xi'  W^p + C'
	\end{equation}
	for all $n \in\N$ for some universal constants $\xi'>0$ and $C'>0$.
	
	For $l\in\N$, recall that a partition $\sigma$ of $[l]$ is defined as a collection of disjoint non-empty sets $\sigma_1,\hdots,\sigma_i\subseteq[l]$ with $i\in\N$, such that $\bigcup_{j=1}^i \sigma_j=[l]$. The sets $\sigma_1,\hdots,\sigma_i$ are called blocks of the partition. By $|\sigma|$ we denote the number of blocks of $\sigma$ (i.e., $|\sigma|=i$), while $|\sigma_j|$ stands for the cardinality of the block $\sigma_j$. Let $\Pi(l)$  be the set of all partitions of $[l]$, while $\Pi_{\ge 2}(l)$ denotes the set of all partitions of $[l]$ with all blocks having size at least two.
	
	Let $p_{x,n,W}:=\E_W \mathds{1}_{x,n}$ for $x \in V_n$, $n \in \N$. Then we have for all $n \in \N, \mathbf{P}\text{-a.s.}$
	\begin{align*}
	\mu_{2k,n,W} :=	&\E_W(D_n-\E_W D_n)^{2k} =  \E_W \left(\sum_{x \in V_n}(\mathds{1}_{x,n}- p_{x,n,W} )\right)^{2k} \\
	=& \sum_{\sigma \in \Pi(2k)} \sum_{\substack{x_1,\dots, x_{|\sigma|} \in V_n\\x_u\not=x_v, u\not=v}} \E_W \prod_{i=1}^{|\sigma|}(\mathds{1}_{x_i,n}- p_{x_i,n,W})^{|\sigma_i|}.
	\end{align*}
	Since the factors are conditionally independent with $\E_W(\mathds{1}_{x,n}-p_{x,n,W})=0$ for all $x \in V_n$, $n \in \N$, the above expectations  vanish for all $\sigma\in\Pi(2k)\setminus\Pi_{\geq 2}(2k)$. Together with the fact that $|\mathds{1}_{x_i,n}-p_{x_i,n,W}|^{|\sigma_i|} \leq (\mathds{1}_{x_i,n}-p_{x_i,n,W})^{2}$ when $|\sigma_i|\geq 2$, this yields that $\mathbf{P}$-a.s.
	\begin{align*}
	\mu_{2k,n,W} & \le \sum_{\sigma \in \Pi_{\ge 2}(2k)} \sum_{\substack{x_1,\dots, x_{|\sigma|} \in V_n\\x_u\not=x_v, u\not=v}} \prod_{i=1}^{|\sigma|} \E_W (\mathds{1}_{x_i,n}- p_{x_i,n,W} )^2 \\
	& \le \sum_{\sigma \in \Pi_{\ge 2}(2k)} \left( \E_W \sum_{x \in V_n}(\mathds{1}_{x,n}-p_{x,n,W})^2 \right)^{|\sigma|} \\
	& = \sum_{\sigma \in \Pi_{\ge 2}(2k)}\left[ \E_W (D_n - \E_W D_n)^2 \right]^{|\sigma|}
	\end{align*}
	for all $n \in \N$. Now, using \eqref{eq:Var} and the fact that $|\sigma| \le k$ for $\sigma \in \Pi_{\ge 2}(2k)$, it follows that
	$$
	\mu_{2k,n,W} \le a_{2k}W^{kp}+C_{2k} \quad \mathbf{P}\text{-a.s.}
	$$
	for all $n \in \N$ for some constants $a_{2k}$ and $C_{2k}$. 	
	This proves the result for all even $m \ge 2$.
	
	Finally, note that for $m=1$, the result is true by \eqref{eq:Var} and the Cauchy-Schwarz inequality. For $m=2k+1$ for some $k \in\N$, the Cauchy-Schwarz inequality and the result for even indices imply that for  some constants $a_{2k+1}$ and $C_{2k+1}$, $\mathbf{P}$-a.s. 
	\begin{align*}
	\E_W|D_n-\E_W D_n|^{2k+1} \le \sqrt{\mu_{2,n,W} \mu_{4k,n,W}} & \le \sqrt{(a_2 W^p + C_2)(a_{4k} W^{2kp} + C_{4k})} \\
	& \le a_{2k+1} W^{\frac{(2k+1)p}{2}} + C_{2k+1}
	\end{align*}
	for all $n \in \N$, which proves the result for all odd $m \ge 1$ concluding the proof.
\end{proof}

\begin{proof}[Proof of Theorem \ref{thm:I-V} (${\rm a}$)-(${\rm c}$)] 
	The condition (A1) is assumed and, as noted above, all the models {\bf I}--{\bf IV} satisfy (A2). The condition (A3) follows from Lemma \ref{lem:A3}. Finally, Lemma~\ref{lem:momg} implies (A4) for models {\bf I} and {\bf II} while it is trivial for models {\bf III} and {\bf IV} as discussed above. Now the result is an immediate consequence of Theorems \ref{thm.PoiLim} and \ref{thm:hillCons}.   
\end{proof}

\subsection{Relating models IV and V}\label{subsec2} As noted in Section~\ref{sec:2}, the models {\bf IV} and {\bf V} bear a lot of similarity. But unlike the other models, in model {\bf V}, there are $n$ vertices in $G_n$ and it has no multiple edges. So the random variable $D_n$ is always less than or equal to $n$, while $W$ can be arbitrarily large, and hence, an assumption like (A3) cannot hold. Thus, we are not able to apply  Theorems \ref{thm.PoiLim} and \ref{thm:hillCons} directly to this model. Instead, we prove Theorem \ref{thm:I-V} (d) by coupling the models {\bf IV} and {\bf V} and using Theorem \ref{thm:I-V} (c).

For each $n\in\N$, we start by constructing a coupling of three graphs $G_n^1$, $G_n^2$, and $G_n^3$ with common vertex set $[n]$ and associated weights $(W_x)_{x\in[n]}$. For $i\in\{1,2,3\}$, let $E_n^i\{x,y\}=E_n^i\{y,x\}$, $x,y\in[n]$, denote the number of edges between the vertices $x$ and $y$ in $G_n^i$. We assume that, for given weights $(W_x)_{x\in[n]}$, the random variables $(E_n^i\{x,y\})_{1 \le x \le y \le n }$ are independent for each $i\in\{1,2,3\}$ with
$$
E_n^1\{x,y\}\sim \mathrm{Poisson}(p'_{xy}),\quad E_n^2\{x,y\}\sim\mathrm{Bernoulli}(p_{xy})
$$
and
$$
E_n^3\{x,y\} \sim \mathrm{Poisson}(p_{xy}),
$$
where
$$
p'_{xy}=W_xW_y/L_n, \quad \text{and} \quad p_{xy}=\min\{W_xW_y/L_n, 1\}
$$
for $x,y\in[n]$. Note that $G_n^1$ and $G_n^2$ are distributed as the random graph models {\bf IV} and {\bf V}, respectively.
Also recall that for $q \in [0,1]$, by the so-called \textit{maximal coupling} which achieves the total variation distance (see e.g.\ \cite[Theorem 2.9]{H16}), one can obtain a coupling $(\hat I, \hat J)$ of $I \sim \mathrm{Bernoulli}(q)$ and $J \sim \mathrm{Poisson}(q)$ so that 
\begin{equation}\label{eq:maxc}
\E|\hat I -\hat J| \le C' q^2
\end{equation}
for some constant $C'$ not depending on $q$. In the following, we assume that $E_n^2\{x,y\}$ and $E_n^3\{x,y\}$ are coupled in this way for all $x,y\in[n]$. Moreover, for all $x,y\in[n]$, we couple $E_n^1\{x,y\}$ and $E_n^3\{x,y\}$ so that $E_n^1\{x,y\}=E_n^3\{x,y\}$ whenever $p_{xy}=p_{xy}'$.

For $i\in\{1,2,3\}$, let $(D_{n,x}^i)_{x \in [n]}$ denote the degree sequence of $G_n^i$ and consider the processes
$$
\mathcal{D}_n^i:=\sum_{x \in [n]} \delta_{\frac{D_{n,x}^i}{q(n)}} \quad\text{and}\quad \mathcal{D}_{k,n}^i:= \frac{1}{k} \sum_{x \in [n]} \delta_{\frac{D_{n,x}^i}{q(n/k)}}
$$
for $k\in\N$.

\begin{proof}[Proof of Theorem \ref{thm:I-V} (${\rm d}$)]
	Recall that we assume (A1) with $\beta>2$. We first prove that as $n \to \infty$,
	\begin{equation}\label{eq:twoproc}
	\mathcal{D}_n^2 \xrightarrow{d} \eta_{\beta} \quad \text{in $M_p((0,\infty])$ } \quad \text{and} \quad \mathcal{D}_{k_n,n}^2 \xrightarrow{d} \nu_{\beta} \quad \text{in $M_+((0,\infty])$.}
	\end{equation}
	For $n \in \N$, let $\mathcal{A}_n$ be the event
	$$
	\mathcal{A}_n=\left\{\max_{1 \le k \le n}W_k^2 \le L_n \right\}.
	$$
	Under $\mathcal{A}_n$, one has $p_{xy}=p'_{xy}$ so that, by construction of the coupling, $G_n^1=G_n^3$. Now, for $q \in (2/\beta,1)$,
	\begin{equation*}
	\Prob{\mathcal{A}_n^c} \le  \mathbf{P}\left\{\max_{1 \le k \le n}W_k^2 >n^q\right\} + \mathbf{P}\left\{L_n<n^q\right\}.
	\end{equation*}
	For the first term in the upper bound, using union bound and (A1), we have
	\begin{align*}
	\limsup_{n\to\infty} \mathbf{P}\left\{\max_{1 \le k \le n}W_k^2 >n^q \right\}  &\le \limsup_{n\to\infty} n\Prob{W^2>n^q}\\
	&= \limsup_{n\to\infty} n^{1-\beta q/2}L(n^{q/2}) = 0.
	\end{align*}
	For the second term, note that since $\beta>2$, by (A1) one has that $\sigma^2=\mathrm{Var}(W)<\infty$. Now using Chebyshev's inequality, letting $\mu=\E W$, we obtain that
	$$\limsup_{n\to\infty} \mathbf{P}\left\{L_n<n^q\right\} \le \limsup_{n\to\infty} \frac{\mathrm{Var}(L_n)}{(n\mu-n^q)^2}= \limsup_{n\to\infty} \frac{n \sigma^2}{(n\mu-n^q)^2}=0,$$
	and hence 
	$$
	\lim_{n \to \infty}\Prob{\mathcal{A}_n^c} = 0.
	$$
	Thus for $\beta>2$, we obtain
	\begin{equation}\label{eq:probcpl}
	\lim_{n \to \infty}\Prob{G_{n}^1\not = G_{n}^3} \le \lim_{n \to \infty}\Prob{\mathcal{A}_n^c} = 0.
	\end{equation}
	Hence from Theorem~\ref{thm:I-V} (c),
	it follows that as $n \to \infty$,
	\begin{equation}\label{eq:new}
	\mathcal{D}_n^3 \xrightarrow{d} \eta_\beta \quad \text{ in }M_p((0,\infty]) \quad \text{and} \quad \mathcal{D}_{k_n,n}^3 \xrightarrow{d} \nu_\beta \quad \text{ in }M_+((0,\infty]).
	\end{equation}
	Next, note that since $\beta>2$, it follows from (A1) that $m_j:=\E W^j<\infty$ for $j \in\{1,2\}$. Let $\mathcal{B}_n$ denote the event
	$$
	\mathcal{B}_{n}=\left\{\Big|\frac{L_n}{n}-m_1\Big| \le \frac{\min\{m_1,m_2\}}{2},\Big|\frac{1}{n}\sum_{i=1}^n W_i^2-m_2\Big| \le \frac{\min\{m_1,m_2\}}{2} \right\}
	$$
	and define
	$$
	D_E=\sum_{x,y=1}^n |E_n^2\{x,y\}-E_n^3\{x,y\}|.
	$$
	Let $(t_n)_{n \in \N}$ be either $(k_n)_{n \in \N}$ or $\mathbf{1}$. Then, using the Markov inequality for $D_E \mathds{1}_{\mathcal{B}_n}$, we obtain
	\begin{align}\label{eq:newsp}
	\Prob{D_E>\sqrt{q(n/t_n)}} \le & \frac{\E[D_E \mathds{1}_{\mathcal{B}_n}]}{\sqrt{q(n/t_n)}} + \mathbf{P}\left\{\Bigg|\frac{L_n}{n}-m_1\Bigg| > \frac{\min\{m_1,m_2\}}{2} \right\}\nonumber \\ & + \mathbf{P}\left\{\Bigg|\frac{1}{n}\sum_{i=1}^n W_i^2-m_2\Bigg| > \frac{\min\{m_1,m_2\}}{2} \right\}.
	\end{align}
	The last two terms in \eqref{eq:newsp} go to zero as $n \to \infty$ by the weak law of large numbers. For the first term, letting $\E_{\bf W}$ denote the conditional expectation given the weights $(W_x)_{x \in [n]}$, using \eqref{eq:maxc} in the first step, we have that $\mathbf{P}$-a.s.
	\begin{align*}
	\E_{\bf W} [D_E \mathds{1}_{\mathcal{B}_n}]  &\le C'\mathds{1}_{\mathcal{B}_n}\sum_{i,j=1}^n\min\{W_iW_j/L_n,1\}^2 \\
	&\le C'\mathds{1}_{\mathcal{B}_n}\frac{\left(\sum_{i=1}^n W_i^2\right)^2}{L_n^2}<C' \frac{9m_2^2}{m_1^2} <\infty.
	\end{align*}
	Hence the first term in \eqref{eq:newsp} also goes to zero as $n \to \infty$. Thus, we obtain
	\begin{equation}\label{eq:t1}
	\lim_{n \to \infty} \Prob{D_E>\sqrt{q(n/t_n)}} =0.
	\end{equation}
	Now note that for any $x \in [n]$, one has that $|D_{n,x}^2-D_{n,x}^3| \le D_E$. Fix $\varepsilon>0$ and $0<a< \infty$. For $\delta>0$ and $n$ large enough so that $1/\sqrt{q(n)}<\delta$, we have
	\begin{align*}
	&\Prob{|\mathcal{D}_n^2((a,\infty])-\mathcal{D}_n^3((a,\infty])| >\varepsilon}\\
	&\le \Prob{\mathcal{D}_n^3((a- D_E/q(n),a+D_E/q(n)] ) \ge 1}\\
	&\le  \Prob{D_E>\sqrt{q(n)}} + \Prob{\mathcal{D}_n^3((a- \delta,a+\delta)) \ge 1}.
	\end{align*}
	On the other hand, for $n$ large enough so that $1/\sqrt{q(n/k_n)}<\delta$, a similar argument yields
	\begin{align*}
	&\Prob{|\mathcal{D}_{k_n,n}^2((a,\infty])-\mathcal{D}_{k_n,n}^3((a,\infty])| >\varepsilon}\\
	&\le  \Prob{D_E>\sqrt{q(n/k_n)}} + \Prob{\mathcal{D}_{k_n,n}^3((a- \delta,a+\delta)) \ge \eps}.
	\end{align*}
	The first terms in the two upper bound above goes to zero by \eqref{eq:t1} as $n \to \infty$. For the second terms, by \eqref{eq:new}, we have
	$$
	\lim_{n \to \infty}\Prob{\mathcal{D}_n^3((a- \delta,a+\delta)) \ge 1}= \Prob{\eta_\beta((a- \delta,a+\delta)) \ge 1},
	$$
	while $\limsup_{n \to \infty} \Prob{\mathcal{D}_{k_n,n}^3((a- \delta,a+\delta)) \ge \eps}$ is bounded by $\mathds{1}\{\nu_\beta((a- \delta,a+\delta)) \ge \eps\}$. Since both tend to zero as $\delta \to 0$, noting that $\varepsilon>0$ was arbitrary, we obtain that for any $0<a< \infty$, as $n \to \infty$,
	\begin{equation*}
	\mathcal{D}_n^2((a,\infty])-\mathcal{D}_n^3((a,\infty]) \xrightarrow{\mathbf{P}}0 \; \text{ and } \; \mathcal{D}_{k_n,n}^2((a,\infty])-\mathcal{D}_{k_n,n}^3((a,\infty]) \xrightarrow{\mathbf{P}}0,
	\end{equation*}
	which in turn imply that for all $0<a \le b \le \infty$, as $n \to \infty$,
	\begin{equation}\label{eq:23}
	\mathcal{D}_n^2((a,b])-\mathcal{D}_n^3((a,b]) \xrightarrow{\mathbf{P}}0 \; \text{ and } \; \mathcal{D}_{k_n,n}^2((a,b])-\mathcal{D}_{k_n,n}^3((a,b]) \xrightarrow{\mathbf{P}}0.
	\end{equation}
	
	Combining \eqref{eq:new} and \eqref{eq:23}, arguing as in the proof of Theorem~\ref{thm.PoiLim} and Proposition~\ref{prop.LimMeasure} by first showing joint convergence in probability for finitely many $\mathcal{I}$-sets and then using \cite[Theorem 1.1]{K73}, we obtain \eqref{eq:twoproc}. The Frechet convergence for the maximum degree of a vertex in $G_n^2$ follows by arguing exactly as in the proof of Theorem~\ref{thm.PoiLim}.
	
	
	Finally we are left to prove the convergence of the Hill estimator for $G_n^2$. Since $\mathcal{D}_{k_n,n}^2 \xrightarrow{d} \nu_\beta$ in $M_+((0,\infty])$ as $n \to \infty$, by Lemma~\ref{lem:hill} with $\xi_n=1$, it follows that
	\begin{equation}\label{eq:lem310}
	\frac{D_{(k_n+1)}^2(n)}{q(n/k_n)} \xrightarrow{\mathbf{P}} 1 \quad \text{as }\; n \to \infty.
	\end{equation}
	For the coupled graphs $G_n^1$, $G_n^2$ and $G_n^3$, let $(D_{(u)}^i)_{u \in [n]}$ denote the order statistics for the degree sequence of $G_n^i$, $i\in\{1,2,3\}$, and let $H_{k_n,n}^i$ be the corresponding Hill estimators, i.e.,
	\begin{equation*}
	H_{ k_n,n}^i = k_n^{-1} \sum _ { u=1} ^ {k_n} \log \frac { D_{(u)}^i(n)} { D_{\left(k_n+1\right)}^i(n)},\; i\in\{1,2,3\}.
	\end{equation*}
	Fix $\varepsilon>0$. Note that there exists an integer $N \in \N$ such that for all $n \ge N$, if 
	$D_E \le \sqrt{q(n/k_n)}$ and $D_{(k_n+1)}^2(n)\geq q(n/k_n)/2$,
	one has 
	$$\Bigg|\log \frac{D_{(u)}^2(n)}{D_{(u)}^3(n)} \Bigg|<\varepsilon/2 \quad \text{for all $1 \le u \le k_n+1$},$$
	so that
	$$|H_{ k_n,n}^2-H_{ k_n,n}^3|=k_n^{-1}\Bigg| \sum_{u=1}^{k_n} \log \frac{D_{(u)}^2(n)}{D_{(k_n+1)}^2(n)}-\sum_{i=1}^{k_n} \log \frac{D_{(u)}^3(n)}{D_{(k_n+1)}^3(n)} \Bigg| \le \varepsilon.$$
	Thus, for $n \ge N$, we have
	$$
	\mathbf{P}\left\{|H_{ k_n,n}^2-H_{ k_n,n}^3|>\varepsilon \right\}\le \Prob{D_E>\sqrt{q(n/k_n)}}+ \mathbf{P}\left\{\Bigg|\frac{D_{(k_n+1)}^2(n)}{q(n/k_n)}-1\Bigg|>\frac{1}{2}\right\}.
	$$
	The first summand in the upper bound tends to zero as $n \to \infty$ by \eqref{eq:t1}, while the second term goes to zero by \eqref{eq:lem310}. As $\eps>0$ was arbitrary, this implies that as $n \to \infty$,
	$$H_{ k_n,n}^2-H_{ k_n,n}^3 \xrightarrow{\mathbf{P}} 0.$$
	On the other hand, by \eqref{eq:probcpl},
	$$\mathbf{P}\left\{H_{ k_n,n}^1 \not = H_{ k_n,n}^3  \right\} \le \Prob{ G_{n}^1\not =  G_{n}^3} \le \Prob{\mathcal{A}_n^c} \to 0$$
	as $n \to \infty$. Finally noting that by Theorem~\ref{thm:I-V} (c), one has $H_{k_n,n}^1 \xrightarrow{\mathbf{P}} 1/\beta$ as $n \to \infty$, Slutsky's theorem yields that as $n \to \infty$,
	$$H_{k_n,n}^2 \xrightarrow{\mathbf{P}} 1/\beta$$
	concluding the proof.
\end{proof}
\section*{Acknowledgements}
The research was supported by the Swiss National Science Foundation Grant No.\ 200021\_175584. Large parts of this paper were written while both authors were employed by the University of Bern. They would also like to thank Federico Pianoforte for some helpful discussions.

\end{document}